\documentclass[11pt]{amsart}
\usepackage{amscd,amsmath,amssymb,amsfonts,ifthen}
 
\usepackage{color} 
\usepackage{graphicx}
\usepackage{url,mathrsfs,euscript}
\usepackage[all]{xypic}
 
\usepackage{charter}
\usepackage{eulervm}

\usepackage{geometry}
 
\begin{document}

\theoremstyle{plain}
\newtheorem{thm}{Theorem}[section]
\newtheorem{lem}[thm]{Lemma}
\newtheorem{cor}[thm]{Corollary}
\newtheorem{prop}[thm]{Proposition}
\newtheorem{conj}[thm]{Conjecture}

\theoremstyle{definition}
\newtheorem{defn}[thm]{Definition}
\newtheorem{question}[thm]{Question}
\newtheorem{nota}[thm]{Notations}
\newtheorem{claim}[thm]{Claim}
\newtheorem{ex}[thm]{Example}
\newtheorem{rmk}[thm]{Remark}
\newtheorem{rmks}[thm]{Remarks}

\def\bar{\overline}
\newcommand{\GG}{\mathbf G}
\newcommand{\aro}{\longrightarrow}
\newcommand{\cn}{\mathcal{N}}
\newcommand{\ot}{\otimes}
\newcommand{\g}[1]{\mathfrak{#1}}
\newcommand{\ep}{\varepsilon} 
\newcommand{\ti}{\times}
\newcommand{\dmod}[1]{\mathcal{D}(#1)\text{-}{\bf mod}}
\newcommand{\si}{\sigma}
\newcommand{\arou}[1]{\stackrel{aro}{#1}}
\newcommand{\wt}{\widetilde}
\newcommand{\de}{\delta}
\newcommand{\cd}{\mathcal D}
\newcommand{\ce}{\mathcal E}
\newcommand{\cm}{\mathcal M}
\newcommand{\cl}{\mathcal L}
\newcommand{\co}{\mathcal O}
\newcommand{\cs}{\mathcal S}
\newcommand{\ct}{\mathcal T}
\newcommand{\cv}{\mathcal V}
\newcommand{\bm}[1]{\boldsymbol{#1}}

\date{May 31, 2017}
\title[Group schemes over DVRs]{On the structure of affine flat group schemes over discrete valuation rings}

\author[N. D. Duong]{Nguyen Dai Duong}
\thanks{The research of the first and the second named authors is supported in part by  Vietnam National Foundation for Science and Technology Development (NAFOSTED), grant number 101.04-2016.19.}

\address{Institute of Mathematics,
Vietnam Academy of Science and Technology.
18 Hoang Quoc Viet, Cau Giay, Hanoi, Vietnam}

\email{nguyendaiduongqn@yahoo.com.vn}

\author[P. H. Hai]{Phung  Ho Hai}
\address{Institute of Mathematics,
Vietnam Academy of Science and Technology.
18 Hoang Quoc Viet, Cau Giay, Hanoi, Vietnam}

\email{phung@math.ac.vn}

\author[J. P. dos Santos]{Jo\~ao Pedro P. dos Santos}

\address{Institut de Math\'ematiques de Jussieu -- Paris Rive Gauche, 4 place Jussieu, 
Case 247, 75252 Paris Cedex 5, France}

\email{joao-pedro.dos-santos@imj-prg.fr}



\begin{abstract}We study affine group schemes over a discrete valuation ring $R$ using two techniques: Neron blowups and Tannakian categories. We employ the theory developed to define and study differential Galois groups of $\mathcal D$-modules on a scheme over a $R$. This throws light on how differential Galois groups of families degenerate.
\end{abstract}

\maketitle

\section{Introduction}\label{05.10.2016--2}The thoughts the reader is about to see in this text come from our will to understand  ``models'' of group schemes defined over a DVR and their Tannakian interpretation. In fact, we set out to identify if a more general theory would be able to accommodate the following two examples.  (For unexplained notations, see the end of this introduction.)

\begin{ex}[$\mathrm{SGA3}$, Exp. $\mathrm{VI}_B$, 13.3]\label{10.07.2014--1}Let $R$ be a discrete valuation ring with uniformiser $\pi$ and quotient field $K$. 
Consider the pro-system of affine group schemes
\begin{equation}\label{05.05.2014--5}
\ldots\longrightarrow \GG_{a,R}\xrightarrow{\times \pi}\GG_{a,R}\xrightarrow{\times \pi}\GG_{a,R},
\end{equation}
which corresponds, on the level of rings, to the inductive system 
\begin{equation}\label{25.05.2015--1}
R[x_0]\longrightarrow R[x_1]\longrightarrow\ldots,\qquad x_i\longmapsto \pi x_{i+1}.
\end{equation}
The limit $G$ of diagram \eqref{05.05.2014--5} is a flat affine group scheme over $R$ whose associated ring $R[G]$ is the colimit of  \eqref{25.05.2015--1}:  $\left\{ P\in K[T]\, | \,P(0)\in R\right\}$. Another meaningful way to express $G$ is by writing it as a join of $R$-schemes: 
\[
G=\GG_{a,K}\underset{e,\mathrm{Spec}\, K}{\amalg}\mathrm{Spec}\,R.
\]
Note that the $R$-module $R[G]$, being isomorphic to $R\oplus K\oplus K\oplus\cdots$ is \emph{not} projective and that $G\otimes R_n$ is always the trivial group scheme over $R_n=R/(\pi^{n+1})$. 
\end{ex}

\begin{ex}[cf. 3.2.1.5 of \cite{Andre}]\label{07.10.2014--4}Let $R$ be a discrete valuation ring with uniformiser $\pi$ and quotient field $K$.   
Let $\GG_{m,R}=\mathrm{Spec}\, R[z,1/z]$. If we set $x_0=z-1$ and $y_0=1/z-1$, so that $R[z,1/z]=R[x_0,y_0]/(x_0+y_0+x_0y_0)$, co-multiplication is  given by 
\begin{equation}\label{28.05.2015--1}
\begin{split}
x_0&\longmapsto  x_0\otimes 1+1\otimes x_0+x_0\otimes x_0\\ y_0&\longmapsto y_0\otimes 1+1\otimes y_0+y_0\otimes y_0.
\end{split}
\end{equation}
Now we write 
\[
G_n=\mathrm{Spec}\,R[x_n,y_n]/(x_n+y_n+\pi^nx_ny_n),
\]
and define co-multiplication by 
\begin{equation}\label{28.05.2015--2}
\begin{split}
x_n&\longmapsto x_n\otimes 1+1\otimes x_n+\pi^nx_n\otimes x_n,\\   y_n&\longmapsto  y_n\otimes 1+1\otimes y_n+\pi^ny_n\otimes y_n.
\end{split}
\end{equation}
It is immediately verified that 
\[\begin{split}x_{n-1}&\longmapsto \pi x_{n}\\ y_{n-1}&\longmapsto \pi y_n\end{split}\] defines a morphism of $R$-algebras $R[G_{n-1}]\to R[G_n]$ which induces an isomorphism $K[G_{n-1}]\to K[G_n]$.
This means that   $R[G_n]$ is none other than the subring of $K[z,1/z]$ generated by $\pi^{-n}x_0$ and $\pi^{-n}y_0$. It is also easily verified that the co-multiplication \eqref{28.05.2015--2} on $R[G_n]$ is obtained from that in \eqref{28.05.2015--1}. We then arrive at the projective system of affine group schemes 
\[\ldots\longrightarrow G_{n+1}\longrightarrow G_{n}\longrightarrow G_{n-1}\longrightarrow\ldots\]
whose limit $G$ is considered in \cite[3.2.1.5]{Andre}. 
Note that, as in Example \ref{10.07.2014--1}, 
\[
R[G]=\{P\in K[z,1/z]\,:\, \varepsilon P\in R\},
\] 
where $\varepsilon$ is the co-identity $z\mapsto1$. Again $G\otimes R_n=\mathrm{Spec}\,R_n$ for all $n$.  
\end{ex}
Furthermore, in \cite{Andre}, the author shows that the group scheme $G$ is a ``differential Galois group''; a fact which kindred our interest in building a heftier theory. 

Our method is to put together the theory of Tannakian categories over a DVR \cite{DH} and the theory of Neron blowups \cite{WW}. Once this is done, two basic principles appear: 
\begin{itemize}\item[\textbf{P1.}]   Tannakian theory over a DVR meets much more frequently group schemes of infinite type. More technically,  ``Galois groups'' or ``images''  can be of infinite type. 
\item[\textbf{P2.}]  The theory of Neron blowups serves to render these inconvenients more tractable. 
\end{itemize}
We then focus on two main tasks: 
\begin{itemize}\item[\textbf{T1.}] Understand to what extent \textbf{P2} is  responsible for the difficulty in \textbf{P1}.
\item[\textbf{T2.}] Detect in  concrete cases when the difficulty in \textbf{P1} is avoided.  
\end{itemize}

Our findings concerning  \textbf{T1}, respectively \textbf{T2}, is in Section \ref{20.08.2015--1}, respectively Section \ref{20.08.2015--2}. We now summarise the contents of the text in more detail. 

   In Section \ref{03.02.2015--3} we introduce and study some basic properties of the Neron blowup of a group scheme over a DVR along a closed subscheme of the special fibre. This is a central technique in studying group schemes over discrete valuation rings. The text relies heavily on the work of Waterhouse and Weisfeiler \cite{WW} and for most of the time (Sections \ref{24.06.2015--1}--\ref{24.06.2015--3}) we simply present their findings in our perspective so that further explanations become more effective. In doing so, we explain how one result from \cite{WW} fits comfortably in general mathematical culture (Theorem \ref{11.02.2015--1}), we elaborate on a hint appearing in \cite{WW} (Proposition \ref{12.08.2015--9}), 
we give a more important role to the notion of standard sequence (Definition \ref{19.11.2014--1}) by including it as part of the structure theorem on generic isomorphisms (Theorem \ref{26.09.2014--6})
and we derive an expression for an arbitrary affine flat group scheme as a limit of flat  \emph{algebraic} ones (Theorem \ref{19.11.2014--1b}). On the other hand, Section \ref{24.06.2015--4} already starts the program mentioned in this introduction: we define automatic blowups of group schemes (Definition \ref{22.08.2014--2}). This is an abstraction of Example \ref{10.07.2014--1} and Example \ref{07.10.2014--4}, and fits into the theory of standard sequences. We then go on to nuance some basic properties of the automatic blowup.

In Section \ref{07.10.2014--3}, we begin to put emphasis on the Tannakian approach and study how to produce faithful representations of Neron blowups. This is motivated by our will to study differential Galois groups. Indeed, differential Galois groups are the reflex  in the Tannakian mirror of certain tensor categories \emph{generated by one object}. Performing blowups modifies the group  on one side and our 
task is to understand how to modify the generating object on the other. We begin by finding faithful representations of Neron blowups of the identity (Proposition \ref{05.05.2014--2}) and then proceed to treat the general case (Proposition \ref{09.10.2014--3}). To repeat, the constructions involved in these last two propositions are tailored to serve a situation where the objects under the lenses is the category of representations, so that manipulations  should take place in there. Once this principle is abandoned, we can offer the reader some simpler results (Proposition \ref{24.06.2015--6} and Proposition \ref{24.06.2015--5}). 

Section \ref{06.03.2015--2} studies possible candidates for images of group schemes. For a morphism $\rho:G\to H$ between affine group schemes over a field $k$, it is well-known that the schematic image $\mathrm{Im}(\rho)$ is a closed subgroup scheme of $H$ and that the natural morphism $G\to\mathrm{Im}(\rho)$ is faithfully flat \cite[Ch. 14]{waterhouse}. (This is a very pleasant feature of this theory.) Said differently, $\rho$ can be written as the composition $\tau\circ \sigma$, with $\sigma$ a closed embedding and $\tau$ faithfully flat.    But, in the setting of group schemes over a DVR, the situation gets a bit more complicated since the aforementioned factorization ceases to exist in general. One is then led to consider factorizations into three morphisms (Definition \ref{21.08.2014--1}), which gives two ``images.'' This is not as simple as the case of affine group schemes over fields, but is not as complicated as that of affine schemes. In possession of these definitions, we then move on to study the behavior of these images under  reduction modulo $\pi$. The structure then becomes more sophisticated and a third ``image''  appears, see diagram \eqref{09.05.2014--1}. While these ``images'' are of course related to each other, they need not be identical (Example \ref{15.10.2014--2}).  We put in evidence some of these relations in Corollary \ref{09.05.2014--2} and Proposition  \ref{24.06.2015--7}. The section ends with a direct translation of some of the results found in terms of the accompanying tensor categories, which are properly defined in Definition \ref{07.04.2015--2} and recognized in Proposition \ref{09.03.2015--2}. 

In Section \ref{05.02.2015--2} we introduce the Neron blowup along a formal closed flat subgroup scheme (Definition \ref{25.09.2014--1}). This is a generalisation of the notion of  automatic blowup introduced in Section \ref{24.06.2015--4} and is the beginning of the march towards understanding how Neron blowups help in controlling ``Galois groups'' of infinite type.  
This section works out some fundamental properties of the new concept and special attention should be paid to Theorem \ref{26.09.2014--7}, which says that the standard sequence of the blowup along a formal closed flat subgroup is ``constant.''
This property is then isolated and studied from an abstract point of view in Section \ref{20.08.2015--1}. There, the goal is  simple: when is a   standard sequence with ``constant'' centres  the standard sequence of a blowup along a formal flat closed subgroup scheme? (The motivation for this question should  be understood in the light of   task \textbf{T1} mentioned earlier since standard sequences give rise to affine group schemes.) We give a partial answer, cf. Corollary \ref{09.04.2015--3}.  

In Section \ref{09.03.2015--9} we apply some of our previous results to differential Galois theory. (As we understand, the object of this theory is the  study of $\mathcal O$-coherent  $\cd$-modules, where $\cd$ is the ring of all differential operators.)  Here we make fundamental use of Tannakian categories over $R$.  The upshot is that for a $\mathcal D$-module  $\cm$ on a smooth scheme $X$  over $R$, one has \emph{two differential Galois groups} in sight: the full $\mathrm{Gal}'(\cm)$ and the restricted $\mathrm{Gal}(\cm)$ (see Definition \ref{21.08.2015--1} and Definition \ref{21.08.2015--2}). This, of course, is a manifestation of the theory developed in Section 
\ref{06.03.2015--2} so that, accordingly,  $\mathrm{Gal}(\cm)$ is always of finite type over $R$, and $\mathrm{Gal}'(\cm)$ can be quite big (see Example \ref{10.11.2016--1}, where  we reinterpret Example \ref{07.10.2014--4}).    
Over the residue field $k$, we arrive at \emph{three group schemes}: the reductions   $\mathrm{Gal}(\cm)\ot k$ and $\mathrm{Gal}'(\cm)\ot k$, and the differential Galois group of the $\cd$-module $\cm|_{X_k}$, $\mathrm{Gal}(\cm\ot k)$.  A simple example where all these groups differ is given (Example \ref{15.10.2014--1}). 

In Section \ref{20.08.2015--2}, we  concentrate on the ``geometric case'' and  suppose that $R$ contains a copy of its residue field. We then present a result, Theorem \ref{absolute-conn}, which moves in the direction of Bolibrukh's theorem on complete integrability and isomonodromic deformations \cite[Theorem 1.2]{sabbah}. 
It  says that for inflated (also called completely integrable) $\mathcal D$-modules over proper and smooth ambient spaces, the aforementioned   differential Galois groups coincide. This is of course an accomplishment of task \textbf{T2} introduced  above. In this same section, and then again on Section \ref{03.03.2015--2}, we also explain a number of technical results concerning the approximation of $R$  by smooth subalgebras (stemming mainly from \cite[$\mathrm{IV}_3$, \S8]{EGA} and \cite{artin_approximation}) which come to be necessary in order to circumvent the fact that  $R$ is not of finite type over the ground field.  

The proof of Theorem \ref{absolute-conn} makes use of two other elaborate webs of ideas which were developed in \cite{EH1}, \cite{dS12} and \cite{Zha14}, and the objective of Section \ref{03.03.2015--2} is to review these works. In fact, the examination we offer of \cite{dS12} and \cite{Zha14} (see Theorem \ref{lem.santos}) only touches upon the main output of these, the exactness of the homotopy sequence.  On the other hand the reading we make of \cite{EH1} is more incisive and as a result, we extract deeper consequences from the methods: compare Theorem \ref{lem.eh} with its inspiration, Theorem 5.10 of \cite{EH1}. 

\subsection*{Acknowledgements} We wish to thank the referee of this journal for an intelligent reading of the work. The third named author thanks Mark Spivakovsky for help with the theory of smoothing of morphisms.

\subsection*{Notations, conventions and standard terminology}\label{notations}

\begin{enumerate}\item Throughout the text, $R$ stands for a discrete valuation ring with quotient field $K$, residue field $k$, and uniformizer $\pi$. All the Hom-sets and tensor products, when not explicitly indicated, are understood to be taken over $R$. The quotient ring $R/(\pi^{n+1})$ is denoted by $R_n$. 
\item Given an object $X$ over $R$ (a scheme, a module, etc), we sometimes find useful to write $X_k$ instead of $X\ot_Rk$, $X_K$ instead of $X\ot_RK$, etc.  
\item To avoid repetitions, by a \textbf{group scheme} over some ring $A$, we understand an \textbf{affine group scheme} over $A$. 
\item The category of group schemes over a ring $A$ will be denoted by $(\mathbf{GSch}/A)$; the full subcategory whose objects are $A$-flat will be denoted by $(\mathbf{FGSch}/A)$. 
\item If $V$ is a free $R$-module of finite rank, we write $\mathbf{GL}(V)$ for the general linear group scheme representing $A\mapsto\mathrm{Aut}_A(V\ot A)$. If $V=R^n$, then $\mathbf{GL}(V)=\mathbf{GL}_{n}$.  
\item 
If $G$ is  a group scheme over $R$, we let $\mathrm{Rep}_R(G)$ stand for the category of representations of $G$ which are, as modules,  \emph{of finite type} over $R$. (We adopt Jantzen's  and Waterhouse's definition of representation. See \cite[Part I,  2.7 and 2.8, 29ff]{jantzen} and \cite[3.1-2, 21ff]{waterhouse}.) 

\item The full subcategory $\mathrm{Rep}_R(G)^o$ of $\mathrm{Rep}_R(G)$ has for objects those $V$ whose underlying $R$-module is free. If $V\in\mathrm{Rep}_R(G)^o$, we  abuse notation and make no graphical distinction between the homomorphism $G\to \mathbf{GL}(V)$ and the coaction $V\to V\ot R[G]$.  
\item An object  $V$ of $\mathrm{Rep}_R(G)^o$ is said to be a faithful representation if the resulting morphism $G\to {\bf GL}(V)$ is a closed immersion. Similar conventions are in force for group schemes over $k$, $R_n$, etc. We admonish the reader that \emph{this is not} the terminology of the authoritative \cite{SGA3}, where a faithful action is decreed to be one having no kernel (see Definition 2.3.6.2 of expos\'e I). Since the literature on the subject is scant and the definition of \cite{SGA3} permits some bizarreries (see Example \ref{07.10.2014--4}), we allow ourselves to make such a modification.   
\item For an affine scheme $X$ over $R$, we let $R[X]$ stand for the ring $\mathcal O(X)$. More generally, if $A$ is any $R$-algebra, we write $A[X]$ to denote $\mathcal O(X\ot_R A)$.
\item If $G$ is a group scheme over $R$, the coidentity or  counit is the morphism $\varepsilon:R[G]\to R$ associated to the identity $\mathrm{Spec}\,R\to G$. The augmentation ideal of $R[G]$ is the kernel of $\varepsilon$.
\item A vector bundle on an locally noetherian scheme $X$ is a locally free  $\mathcal O_X$-module of finite rank.  
\item If $M\subset N$ is an inclusion of $R$-modules, we write $M^\mathrm{sat}$ to stand for the saturation of $M$ in $N$, i.e. $\cup_m(M:\pi^m)$. 
\end{enumerate}

\section{Introducing Neron blowups}\label{03.02.2015--3}

\subsection{Definition and basic properties}\label{24.06.2015--1}

Let $G\in (\textbf{FGSch}/R)$.  In what follows we regard $R[G]$ as a subring of $K[G]$. 

Let $H_0$ be a closed subgroup of  $G_k$ defined by an ideal $J$. Note that, in this case, $\pi\in J$. 
The \emph{Neron blowup} $G'$ of $G$ at $H_0$  (cf. \cite[p.551]{WW} and \cite[2.1.2]{Anan}) is the spectrum of the subring $R[G']$ of $K[G]$ generated by $R[G]$ and the elements of $\pi^{-1}J$. (It is obviously sufficient to adjoin elements $\pi^{-1}f$, where $f$ runs over a system of generators of $J$.)  
The natural morphism $R[G]\to R[G']$ is a morphism of Hopf algebras over $R$ as is readily verified. 
(See also \cite[Proposition 1.1]{WW}.) By construction, $G_K'\to G_K$ is an isomorphism  and $G'_k\to G_k$ factors through $H_0$. In fact we have:

\begin{lem}[Universal property, see \cite{WW}, p.551]\label{09.04.2015--1} Let $\mathcal G\to G$ be a morphism in $\mathbf{FGSch}/R$ such that $\mathcal G_k\to G_k$ factors through $H_0\to G_k$. Then there exists a unique $\mathcal G\to G'$ rendering the diagram 
\[
\xymatrix{\mathcal G\ar[dr]\ar[r]& G'\ar[d]\\ &G}
\] 
commutative. \qed
\end{lem}

\begin{rmk}\label{26.09.2014--4}
It should be noted that this very elementary definition of blowup in fact relates to the usual notion as follows. Let $(\pi,a_1,\ldots,a_r)\subset R[G]$ be an ideal  cutting out on $G_k$ a closed subgroup scheme $H_0$. Write $G'$ for its Neron blow-up. Then 
\[
\frac{R[G][\xi_1,\ldots,\xi_r]}{(\pi\xi_1-a_1,\ldots\pi\xi_r-a_r)^\mathrm{sat}}\longrightarrow K[G],\quad \xi_i\longmapsto \frac{a_i}{\pi},
\]
defines an isomorphism onto $R[G']$. As is well-known, in case $R[G]$ is noetherian, the ring on the left above is the ring of an affine piece of the blowup \cite[Lemma 1.2(e), p.318]{liu}.
\end{rmk}

\subsection{A result stemming directly from the theory of blowups}\label{24.06.2015--2}
We use this section to state a generalization of a result appearing in \cite{WW} and which is more transparently treated by using (nowadays) commonplace algebraic geometry. We begin by recalling the following important theorem,  whose proof the reader can find   in \cite[Ch. X,\S9,no.7]{BA}.  
\begin{thm}\label{05.10.2016--1}Let $\mathcal I\subset\mathcal O_X$ be an ideal of the locally noetherian scheme $X$. Let $\psi:\widetilde X\to X$ stand for the blowup of $\mathcal I$. Let $U=\mathrm{Spec}\,A$ be an affine open of $X$ where $\mathcal I$ is generated by a regular $A$-sequence $x_1,\ldots,x_c$. 
Then the $U$-scheme $\psi^{-1}(U)$ is isomorphic to \[\mathrm{Proj}\,\frac{A[\xi_1,\ldots,\xi_c]}{(\xi_ix_j-\xi_jx_i)}.\]  
 \qed
\end{thm}

From this and Theorem 1.5 of \cite{WW}, we arrive at the following generalization of \cite[Theorem 1.7]{WW}. (Recall from [EGA $\mathrm{IV}_4$, 17.12.1, p.85] that a closed immersion of smooth schemes is always regular.)

\begin{thm}\label{11.02.2015--1}Let $G\in(\mathbf{FGSch}/R)$ be of finite type. Let $H_0\subset G_k$ be a regularly immersed  closed sub-group scheme. Write $c$ for the  codimension of $H_0$ in $G_k$. If $G'\to G$ stands for the Neron blowup of $G$ at $H_0$, then $G'\ot k\to  H_0$ is a smooth surjective morphism whose kernel  is isomorphic to $\GG_{a,k}^{c }$.  \qed
\end{thm}

\begin{ex}Let $k$ be of positive characteristic $p$. Let $G=\GG_{a,R}$ and denote the Neron blowup of $\alpha_p\subset G_k$ by $G'\to G$. Then $G'_k\to\alpha_p$ is a smooth morphism with kernel $\mathbf G_{a,k}$. 
\end{ex}

\subsection{The action of the centre of the Neron blowup on the kernel}\label{12.08.2015--10}
Let $G\in(\mathbf{FGSch}/R)$ be of finite type. Let $H_0\subset G_k$ be a regularly immersed  closed subgroup scheme. From Theorem \ref{11.02.2015--1} we know that $G_k'\to H_0$ is a faithfully flat morphism so that $H_0$ acts by group scheme automorphisms on  the \emph{right of} the kernel $F$. This is explained in \cite[Ch. III, \S6, no.1, 431ff]{DG}, but is really a triviality: ``lift and conjugate.'' 
Since $F\simeq \mathbf G_{a,k}^c$,  then one might ask if this action is linear and if it can be obtained from something simpler. The answer is Proposition \ref{12.08.2015--9} below, which is the theme of this section. 

To grasp the meaning of Proposition \ref{12.08.2015--9}, we   need the notion of the \emph{conormal} representation. Let $\mathcal G$ be an affine group scheme of finite type over $k$ and let $\mathcal H\subset\mathcal  G$ be a closed subgroup cut out by the ideal $I\subset k[\mathcal G]$. We let $\mathcal H$ act on $\mathcal G$ \emph{on the right} by conjugation: $g\cdot h=h^{-1}gh$. In this way, we obtain an $\mathcal H$-module structure on $k[\mathcal G]$. If $\g a\subset k[\mathcal G]$ stands for the augmentation ideal, it is easily seen that $\g a,I$  and $\g aI$ are all sub-$\mathcal H$-modules of $k[\mathcal G]$. Note that $I/\g aI$ is a finite dimensional $k$-space. 
\begin{defn}The  $\mathcal H$-module $I/\g aI$ is called the co-normal representation. It will be denoted in this text by $\nu:\mathcal H\to\mathbf{GL}(I/\g aI )$. 
\end{defn}

Obviously, $I/\g aI$ is simply the fibre at $e$ of the conormal sheaf   of the immersion $\mathcal H\subset\mathcal G$ [EGA $\mathrm{IV}_4$, p.5]. We invite the reader to explicate the relation between $\nu$ and the pertaining adjoint representations.

\begin{prop}\label{12.08.2015--9}We maintain the setting of the first paragraph of this section. Then there exists an isomorphism of group schemes 
\[F\simeq \mathbf G_{a,k}^c\simeq \mathrm{Spec}\,k[\xi_1,\ldots,\xi_c]\] 
such that the action of $H_0$ on $F$ is determined by 
\[\xi_j\longmapsto \sum_{i=1}^c  \xi_i\ot \nu_{ij},  \] 
where $[\nu_{ij}]\in\mathrm{GL}_c(k[H_0])$ defines the conormal representation of $H_0$. 
\end{prop}

\begin{proof}In the argument, we employ some useful notations: 
\begin{itemize}\item We write ``$\times$'' for the fibre product over $R$ and note that if $X$ and $Y$ are $k$-schemes, then $X\times Y$ is canonically isomorphic to $X\times_{\mathrm{Spec}\,k}Y$  [EGA $\mathrm{I}$, 3.2.4, p.105].
\item If $A$ and $B$ are $R$-algebras  and $\g A$ is an ideal of $A$, we write $\langle \g A\ot1\rangle$ for the ideal generated by the image of $\g A\ot_RB\to A\ot_RB$.  
\item If $\g p$ is a prime ideal of a ring $A$ and $I\subset \g p$ is a sub-ideal, we write $I_{\g p}$ for $IA_{\g p}$. 
\end{itemize}

Let $\gamma:G\ti  G\to G$ stand for the morphism defined by $(x,g)\mapsto g^{-1}xg$; it produces a right action of $G$ on itself. Letting $\gamma':G'\ti G'\to G'$ stand for the analogous action of $G'$ on itself,  we arrive at a commutative diagram
\[\xymatrix{G'\ti G'\ar[r]^-{\gamma'}\ar[d]&G'\ar[d]\\ G\ti G\ar[r]_-{\gamma}&G.}\]
Let $V\subset G$ be \emph{any} affine open subset whose intersection with $H_0$ is the source of a scheme-theoretic section to $G_k'\to H_0$:
\[\sigma:H_0\cap V\aro G_k'.\]
(The existence of $V$ is a consequence of the description of the blowup offered in Theorem \ref{05.10.2016--1}.) By definition of the right action 
\[\alpha:F\times H_0\longrightarrow F\] we have a commutative diagram 
\begin{equation}\label{11.08.2015--1}
\xymatrix{   G_k'\ti G_k'\ar[rr]^-{\gamma_k'}&& G_k'    
\\ F\ti G_k'\ar@{^{(}->}[u]\ar[rr]&& F\ar@{^{(}->}[u] \\
 F\ti(H_0\cap V)\ar[u]^{\mathrm{id}\times\sigma}\ar[rru]_-{\alpha}.&&    }
\end{equation}
Let $\g n$ be a closed point of  $H_0\cap V$ and complete diagram \eqref{11.08.2015--1} as follows: 
\begin{equation}\label{12.08.2015--8}\xymatrix{  
&G_k'\ti G_k'\ar[rr]^-{\gamma_k'}&& G_k'    
\\ 
 &F\ti G_k'\ar@{^{(}->}[u]\ar[rr]&& F\ar@{^{(}->}[u] 
 \\ 
F\ti\mathrm{Spec}\,k[H_0]_{\g n}\ar[r]& F\ti(H_0\cap V)\ar[u]^{\mathrm{id}\ti\sigma}\ar[rru]_-{\alpha}.&&    }
\end{equation}
We will  show that the obvious composition 
\[F\ti \mathrm{Spec}\, k[H_0]_{ \g n}\aro F\ti (H_0\cap V)\stackrel{\alpha}{\longrightarrow} F\]
coincides with the composition 
\[F\ti \mathrm{Spec}\, k[H_0]_{\g n}\aro F\ti (H_0\cap V)\stackrel{\nu}{\longrightarrow} F,\]
where $\nu$ is the conormal representation. 

As usual, let $\varepsilon:R[G]\to R$ be the co-identity and write $\g a$ for its kernel. Write $\g m$ for the ideal $(\pi,\g a)$ --- it corresponds to the identity on the closed fibre $G_k$ --- and note that $\varepsilon$ extends to $R[G]_{\g m}$. It goes without saying that the kernel of $\varepsilon:R[G]_{\g m}\to R$ is just $\g aR[G]_{\g m}$.

\begin{lem}\label{30.11.2016--1}There exist functions $x_{1,\g m},\ldots,x_{c,\g m}\in R[G]_{\g m}$ such that 
\begin{enumerate}\item for each $j$, $\varepsilon(x_{j,\g m})=0$, and 
\item the sequence $\pi,x_{1,\g m},\ldots,x_{c,\g m}$ is regular and generates the ideal of $H_0$ at $R[G]_{\g m}$. 
\end{enumerate}  
\end{lem}

\begin{proof} Let $\pi,x_{1,\g m}^*,\ldots,x_{c,\g m}^*$ be a regular sequence generating the ideal of $H_0$ at $\g m$. 
We can always write $x_{j,\g m}^*=c_j+x_{j,\g m}$, with $c_j\in R$ and $x_{j,\g m}\in \g aR[G]_{\g m}$. Moreover, we know that $\varepsilon(x^*_{j,\g m})\equiv0\mod\pi$, since $x_{j,\g m}^*+(\pi)\in k[G]$ belongs to the ideal of $H_0$. It then follows that $c_j\equiv 0\mod\pi$. Write $c_j^*:=\pi^{-1}c_j$. Then $x_{j,\g m}=x_{j,\g m}^*-\pi c_j^*$,
which means that we have an equality of ideals 
\[(\pi,x_{1,\g m},\ldots,x_{c,\g m})=(\pi,x_{1,\g m}^*,\ldots,x_{c,\g m}^*).
\]
The proof is finished by EGA $\mathrm{IV}_4$, 16.9.5, p.47. 
\end{proof}

The morphism $\gamma$  induces an arrow between local rings 
\[\gamma^\#: R[G]_{\g m}\aro R[G]_{\g m}\ot_R R[G]_{\g n}.\]
Since $H_0$ is a subgroup of $G_k$, if $J\subset R[G]$ stands for its ideal, it follows that  
\[
\gamma^\#(J_{\g m})\subset  \langle J_{\g m}\ot1\rangle+\langle1\ot J_{\g n}\rangle.  
\]
If $y_{1,\g n},\ldots,y_{c,\g n}$ denote elements of $R[G]_{\g n}$ such that $\pi, y_{1,\g n},\ldots,y_{c,\g n}$ forms a regular sequence generating $J_{\g n}$,  we conclude that 
\begin{equation}\label{12.08.2015--1}
\gamma^\#(x_{j,\g m})=\sum_i a_{ij}{x_{i,\g m}} \ot a_{ij}'+\sum_ib'_{ij}\ot b_{ij}{y_{i,\g n}} +\pi s_j
\end{equation}
for $a_{ij},b'_{ij}$ in $R[G]_{\g m}$,  $a_{ij}',b_{ij}$ in $R[G]_{\g n}$ and $s_j$ in $R[G]_{\g m}\ot_RR[G]_{\g n}$.

Since  $\varepsilon(x_{j,\g m})=0$ (see Lemma \ref{30.11.2016--1}), eq. \eqref{12.08.2015--1} gives 
\[0 = \sum_i\varepsilon(b_{ij}')\cdot b_{ij}y_{i,\g n}+ \pi\cdot\left( \varepsilon\ot_R\mathrm{id}_{R[G]_{\g n}}\right) (s_j) ,
\]
which proves that 
\[\pi  \cdot\left( \varepsilon\ot_R\mathrm{id}_{R[G]_{\g n}}\right) (s_j)\in (\boldsymbol y_{\g n}).\]
Since the ideal $(\boldsymbol y_{ \g n})$ has no $\pi$-torsion, because the sequence $\{y_{1,\g n},\ldots,y_{c,\g n},\pi\}$ is regular, it follows that 
\[\left( \varepsilon\ot_R\mathrm{id}_{R[G]_{\g n}}\right) (s_j)\in (\boldsymbol y_{\g n}).\]
Hence, abusing notation and confusing $\g a$ with $\g aR[G]_{\g m}$, 
\begin{equation}\label{12.08.2015--2}s_j\in \langle\g a\ot_R1\rangle+\langle1\ot_R(\boldsymbol y_{\g n})\rangle .
\end{equation}
Since the sequence $\{\pi,x_{1,\g m},\ldots,x_{c,\g m}\}$ is regular, it follows that we have an isomorphism of $R[G]_{\g m}$-algebras:
\begin{equation}\label{12.08.2015--4}R[G']_{\g m}=\frac{R[G]_{\g m}[\boldsymbol\xi]}{(\pi\boldsymbol\xi-\boldsymbol x_{\g m})}.\end{equation}
(Here $R[G']_{\g m}$ stands for the ring obtained by inverting the elements of $R[G']$ in the image of $R[G]\setminus\g m$.)
Now, using (EGA $\mathrm{IV}_1$, 15.2.4, p.21), there exists some affine neighbourhood $V$ of $\g n$ in $G$ and functions   $\boldsymbol y\in R[V]$ inducing $\boldsymbol y_{\g n}$  such that  $\{\pi,y_1,\ldots,y_c\}$ is $R[V]$-regular. Consequently,  if $G'|_V$ stands for the inverse image of $V$ in $G'$, we have 
\begin{equation}\label{12.08.2015--5}
R[G'|_V]\simeq\frac{R[V] [\boldsymbol\eta]}{(\pi\boldsymbol\eta-\boldsymbol y )}
\end{equation}
as $R[V]$-algebras. 
In particular, $V\cap H$ is the source of a section  
$\sigma: H_0\cap V\to G'_k$ defined on the level of rings by $\boldsymbol\eta\mapsto\boldsymbol0$. 

Now, $\gamma'^{\#}$ fits into a diagram 
\[
\xymatrix{
R[G']_{\g m} \ar[rr]^-{\gamma'^{\#}} && R[G']_{\g m}\ot_RR[G']_{\g n} \\ R[G]_{\g m}\ar[rr]_-{\gamma^\#}\ar[u]&&R[G]_{\g m}\ot_RR[G]_{\g n}.\ar[u] 
}\]
Then, from eqs. \eqref{12.08.2015--1} and \eqref{12.08.2015--4}, we have  
\[
\pi\gamma'^{\#}(\xi_j)=\pi\sum_i a_{ij}\xi_{i}\ot a_{ij}'+\pi\sum_ib_{ij}'\ot b_{ij}\eta_i+\pi s_j,
\]
so that 
\begin{equation}\label{12.08.2015--3}
\gamma'^{\#}(\xi_j)=\sum_i a_{ij}\xi_{i}\ot a_{ij}'+\sum_ib_{ij}'\ot b_{ij}\eta_i+  s_j.
\end{equation}
Consequently,   eqs. \eqref{12.08.2015--2} and \eqref{12.08.2015--5} show  that 
\begin{equation}\label{12.08.2015--7} \gamma'^{\#}(\xi_j)\equiv\sum_i \varepsilon(a_{ij})\xi_i\ot a_{ij}'  \mod\, (\pi)+\langle\g a\ot_R1\rangle+\langle 1\ot_R(\boldsymbol \eta) \rangle   .
\end{equation}
Now  we write diagram \eqref{12.08.2015--8} on the level of rings and complete :  
\[\xymatrix{&&&&k[G']_{\g m}\ar @/_1pc/@{-->}[lllld]\ar@/^1pc/@{-->>}[ddl]
\\
 k[G']_{\g m}\ot k[G']_{\g n}\ar@{-->>}[d] &\ar [l] \ar@{->>}[d] k[G']\ot k[G']&&\ar[ll]_-{\gamma_k'^{\#}} k[G'] \ar@{->>}[d]    \ar@{-->}[ru]
\\ 
k[F]\ot k[G']_{\g n}\ar[d]_{\text{kill $1\ot\boldsymbol\eta$}}    & \ar[l] \ar[d]_{\mathrm{id}\ot\sigma^\#}k[F]\ot k[G'] &&\ar[ll] k[F] \ar[lld]^{\alpha^\#}
\\
k[F]\ot k[H_0]_{\g n}   &\ar[l] k[F]\ot k[ H_0\cap V] .&&    }
\]
It then follows from eq. \eqref{12.08.2015--7} that the image of $\xi_j$ via  
\[k[F]\stackrel{\alpha^\#}{\aro}k[F]\ot k[H_0\cap V]\aro  k[F]\ot k[H_0]_{\g n}\]
 is $\sum_i \varepsilon(a_{ij})\xi_i\ot a_{ij}'$.
 By definition, this is the co-normal representation $\nu$. 
\end{proof}

\begin{question}Is a regularly immersed closed affine algebraic subgroup scheme cut out by a regular sequence?
\end{question}

\begin{rmk}Proposition \ref{12.08.2015--9} is  mentioned on p. 554 of \cite{WW}. 
\end{rmk}

\subsection{The standard blowup sequence}\label{24.06.2015--3}

As remarked in \cite{WW}, Neron blowups allow us to decompose \emph{any} morphism of group schemes which   is an isomorphism on generic fibres. 

\begin{thm}[cf. Theorem 1.4 of \cite{WW}]\label{26.09.2014--6}Let $\rho:\mathcal G\to G$ be an arrow of $\mathbf{FGSch}/R$ which is an isomorphism on generic fibres. Write $\rho_0=\rho$ and $G_0=G$. Let $\rho_n:\mathcal G\to G_n$ be constructed and define $G_{n+1}$, respectively $\rho_{n+1}$, as the  blowup of $\mathrm{Im}(\rho_n\ot k)$, respectively the natural morphism $\mathcal G\to G_{n+1}$. Then $\varprojlim_n\rho_n:\mathcal G\to \varprojlim G_n$ is an isomorphism.  

\end{thm}
\begin{proof}For the convenience of the reader and to avoid the hypothesis of finite generation present in \cite{WW}, we summarize the proof.

We prove by induction that $\pi^{-n}R[G_0]\cap R[\mathcal G]\subset R[G_n] $. As the ideal of $R[G_0]$ cutting out $\mathrm{Im}(\rho_0\ot k)$ is $R[G_0]\cap\pi R[\mathcal G]$, the case $n=1$ is readily proved. We now assume that $\pi^{-n}R[G_0]\cap R[\mathcal G]\subset R[G_n]$. Let $f_0\in R[G_0]$ be such that $\pi^{-n-1}f_0$ lies in $R[\mathcal G]$. Then $\pi^{-n}f_0$ belongs to $\pi R[\mathcal G]$. On the other hand, the inclusion $\pi^{-n}R[G_0]\cap R[\mathcal G]\subset R[G_n]$ forces $\pi^{-n}f_0$ to be in $R[G_n]$, so that $\pi^{-n}f_0$ belongs to $R[G_n]\cap \pi R[\mathcal G]$. The latter ideal cuts out $\mathrm{Im}(\rho_n\ot k)$, and therefore $\pi^{-1}\pi^{-n}f_0$ lies in $R[G_{n+1}]$.

Using the above inclusion, we arrive at $R[\mathcal G]\subset \cup R[G_n]$. Since each $R[G_n]$ is contained in $R[\mathcal G]$, the proof of finished.
\end{proof}
\begin{defn}\label{19.11.2014--1}The sequence of morphism described in the statement of Theorem \ref{26.09.2014--6} is  called the standard blowup sequence associated to $\mathcal G\to G$. If context prevents any misunderstanding, we simply speak of the standard sequence of $\mathcal G\to G$.
\end{defn}

The concept of standard sequence also allows the following point of view \cite[p.552, Remarks]{WW}. 
\begin{defn}\label{03.02.2015--2}A diagram 
\[
\cdots\stackrel{}{\longrightarrow}G_{n+1}\stackrel{\rho_n}{\longrightarrow}\cdots\stackrel{\rho_0}{\longrightarrow}G_0
\]
in $\textbf{FGSch}/R$ is a standard sequence  if \begin{enumerate}\item each arrow $\rho_n$ is a blowup of some closed subgroup $B_n\subset G_n\ot k$ and \item  the induced morphism  $B_{n+1}\to B_n$ is faithfully flat.
\end{enumerate} 
\end{defn}
The overburdening of the term ``standard sequence'' caused by Definitions \ref{19.11.2014--1} and Definition \ref{03.02.2015--2} is ephemeral in view of:   
\begin{lem}\label{03.03.2015--1}Let 
\[
\cdots\stackrel{}{\longrightarrow}G_{n+1}\stackrel{\rho_n}{\longrightarrow}\cdots\stackrel{\rho_0}{\longrightarrow}G_0
\] be a standard sequence as in Definition \ref{03.02.2015--2}. Then $\varprojlim_nG_n\to G_0$ is an isomorphism on generic fibres and its standard blowup sequence is the above diagram.\qed
\end{lem}

We shall apply Theorem \ref{26.09.2014--6} to 
give a description of   affine group schemes 
over $R$, see Theorem \ref{19.11.2014--1b} 
below. 
First we 
give a definition.
\begin{defn}\label{21.04.2015--1}
Let $H'$ be a flat Hopf algebra over $R$. A Hopf subalgebra $H$ of $H'$ is an $R$-submodule equipped with a Hopf algebra structure such that the inclusion $H\to H'$ is a homomorphism of Hopf algebras. We say that $H$ is \emph{saturated} in $H'$ if $H'/H$ is flat as an $R$-module.
\end{defn}
\begin{rmks} \label{21.04.2015--2}\begin{enumerate}
\item The image of a Hopf algebra homomorphism is a Hopf subalgebra (of the target).
\item One can always saturate a Hopf subalgebra. Adopting the notations explained at the end of Section \ref{05.10.2016--2} and  Definition \ref{21.04.2015--1}, we put  
\[
\Delta_{H^{\rm sat}}(h):=\pi^{-m}\Delta_{H}(\pi^mh),\quad \text{if $\pi^mh\in H$}.
\]
It  can be extracted from the proof of Lemma 3.1.2 in \cite{DH} that this is a coproduct for $H^{\rm sat}$.
\item Some authors define Hopf subalgebras as our saturated Hopf subalgebras. But Neron blowups justify our definition, see also Theorem \ref{19.11.2014--1b}. 
\end{enumerate}
\end{rmks}

\begin{thm}\label{19.11.2014--1b}
Let $G$ be a flat   group scheme over   $R$. Then we can present it as the the limit of a pro-system of flat   group schemes
\[G:=\varprojlim_i G_{i},\]
in which all morphisms are faithfully flat and the generic fiber of each $G_i$ is of finite type over $K$. Further, each $G_i$ can be obtained from a flat group scheme of finite type by a standard sequence.
\end{thm}
\begin{proof} Consider $R[G]$ as a (right) 
regular representation of $G$ (i.e. as comodule 
on itself by means of the coproduct).
According to Serre \cite{serre}, $R[G]$ is the 
union of its sub-representations, which are 
finite as $R$-modules. Let $V$ be such a sub-representation (note that $V$ is free over $R$). Then $V$ 
induces a Hopf algebra homomorphism 
$R[GL(V)]\to R[G]$, the image of which is a 
Hopf subalgebra of $R[G]$, denoted by 
$R[G(V)]$. Notice that $V$ is a subset of 
$R[G(V)]$. Thus $R[G]$ is a union of its Hopf 
subalgebras $R[G(V)]$. One can take the $G_i$ 
in Theorem to be the saturation of the 
$R[G(V)]$ as $V$ runs in the set of finite sub-representations of $R[G]$.

Now by means of Theorem \ref{26.09.2014--6}, the saturation $R[G(V)]^\text{sat}$ is obtained from $R[G(V)]$ by iterated Neron blowups.
\end{proof}

\subsection{Study of a particular case}\label{24.06.2015--4}
Examples \ref{10.07.2014--1} and \ref{07.10.2014--4} are a particular case of the following process. Let $G$ be a flat group scheme over $R$.

\begin{defn}\label{22.08.2014--2}  The automatic blowup \[N\aro G\] of the identity is the limit of the diagram 
\[\cdots \longrightarrow G_n\longrightarrow\cdots\longrightarrow G_0\] 
where $G_0=G$ and $G_{n+1}\to G_n$ is the Neron blowup of $\{e\}\subset G_{n}\ot k$.
\end{defn}

\begin{rmk}The construction put forth in the above definition will be generalized in Section \ref{05.02.2015--2} below. 
\end{rmk}

\begin{prop}\label{05.02.2015--1}Write $\g a$ for the augmentation ideal of $R[G]$. Then, the algebra $R[N]$ of the automatic blowup of the identity is the $R[G]$-subalgebra of $K[G]$ generated by $\cup_n\pi^{-n}\g a$. The latter union also generates the augmentation ideal of $R[N]$. 
\end{prop}
\begin{proof}
We let $G_n$ be as in Definition \ref{22.08.2014--2} and write $\g a_n$   for the augmentation ideal of $R[G_n]$. It is easily proved that $(\pi^{-1}\g a_n)=\g a_{n+1}$. Let us now show by induction that $R[G_0][\pi^{-n}\g a_0]=R[G_n]$ and $\g a_n=(\pi^{-n}\g a_0)$. As there is nothing to be checked for $n=0$, we assume that $R[G_0][\pi^{-n}\g a_0]=R[G_n]$ and $\g a_n=(\pi^{-n}\g a_0)$. Now, if $A\subset K[G]$ is any $R$-algebra and $S$ is any subset of $A$ generating an ideal $\g S$, it is clear that $A[\pi^{-1}\g S]=A[\pi^{-1}S]$. Hence, $R[G_{n+1}]=R[G_n][\pi^{-1}\g a_n]$ is just $R[G_{n}][\pi^{-n-1}\g a_0]$, which equals $R[G_0][\pi^{-n-1}\g a_0]$. This proves the first claim. As $\g a_{n+1}=(\pi^{-1}\g a_n)$ and $\g a_n=(\pi^{-n}\g a_0)$, we conclude that $\g a_{n+1}=(\pi^{-n-1}\g a_0)$, and the second claim follows as well.  
\end{proof}

Some noteworthy properties present in the examples remain valid. Others easily come forward. The next statement employs the notion of fibre product of rings, cf. \cite[Section 1]{ferrand}.

\begin{cor}\label{06.02.2015--2}If $\varepsilon_K:K[G]\to K$ stands for the  co-identity, then the first projection 
\[
K[G]\times_{\varepsilon_K, K}R\aro K[G]
\]
induces an isomorphism onto $R[N]$. Geometrically, we have an isomorphism of $R$-schemes 
\[
N\simeq G_K\underset{e,\,\mathrm{Spec}\,K}{\coprod}\mathrm{Spec}\,R. 
\]
\end{cor}
\begin{proof}It is clear that the first projection is an isomorphism onto the subalgebra $A:=\{f\in K[G]\,:\,\varepsilon_K(f)\in R\}$. Equally clear is the inclusion $R[N]\subset A$, so that we are left with the verification of $A\subset R[N]$. 
If $f\in A$, then $f=c+f'$ with $c\in R$ and $\varepsilon_K(f')=0$; this means that $f'=\pi^{-m}f''$ with $f''$ in the augmentation ideal of $R[G]$. By Proposition \ref{05.02.2015--1}, $f'\in R[N]$ and   $A\subset R[N]$ is verified.

The proof of the final statement follows directly from the first and from the fact that the functor $\mathrm{Spec}$ sends the cartesian diagram of rings in sight into a co-cartesian diagrams of schemes \cite[Theorem 5.1, p. 568]{ferrand}. 
\end{proof}

\begin{cor}\label{06.02.2015--1}Let $\mathcal G\to G$ be a morphism of flat group schemes over $R$ which induces an isomorphism on generic fibres.  Then there exists a unique  arrow $N\to G$ rendering commutative  the diagram 
\[\xymatrix{  N\ar[r]\ar[d]&\mathcal G\ar[dl] \\ G.& }\]
Said differently, $N\to G$ is an initial object in the category of flat group schemes $\mathcal G\to G$ over $G$ which induce an isomorphism on generic fibres. 
\end{cor}

\begin{proof}We need to prove that $R[\mathcal G]\subset K[G]$ is contained in $R[N]$. For that, let $\g a$ and $\g a_{\mathcal G}$ stand respectively for the  augmentation ideals of $R[G]$ and $R[\mathcal G]$. Writing an $f\in \g a_{\mathcal G}$ as $\pi^{-m}f'$ with $f'\in \g a$, we conclude that $f\in \cup_n\pi^{-n}\g a$. But then $f\in R[N]$ (Proposition \ref{05.02.2015--1}) and we are done since $R\cdot1\oplus\g a_{\mathcal G}=R[\mathcal G]$. 
\end{proof}

\begin{cor}\label{25.02.2015--1}Let $N\to G$ be the automatic blowup of the identity. 
Then, for each $n\ge0$, the group scheme $N\ot R_n$ over $R_n$ is trivial, i.e. isomorphic to $\mathrm{Spec}\,R_n$. 

Conversely, let $\rho:\mathcal G\to G$ be an arrow of $\mathbf{FGSch}/R$ which is an isomorphism on generic fibres. If $\mathcal G\ot R_n$ is trivial for all $n$, then $\rho$ is the automatic blowup of the identity.
\end{cor}
\begin{proof}In this argument, we let $\g a$, $\g a_N$ and $\g a_{\mathcal G}$ stand respectively for the augmentation ideals of $R[G]$, $R[N]$ and $R[\mathcal G]$. 
We know that $R[N]=\cup_nR[G][\pi^{-n}\g a]$ and that $\g a_N=(\cup_n\pi^{-n}\g a)$. Hence, $\g a_N\subset(\pi^{n+1})$ for any $n$; the first claim follows.  On the other hand if $R_n[\mathcal G]$ is always trivial, then $\g a_{\mathcal G}$ is contained in $(\pi^{n+1})$. Hence, $\g a\cdot R[\mathcal G]\subset (\pi^{n+1})$, which shows that $R[\mathcal G]\supset\cup_nR[G][\pi^{-n}\g a]=R[N]$. In view of Corollary \ref{06.02.2015--1}, the inclusion $R[N]\supset R[\mathcal G]$ always holds and the proof is finished. 
\end{proof}

In fact, the second statement in the above corollary allows the following relevant amplification in case $G$ is a closed subgroup scheme of $\mathbf{GL}_{n,R}$. 

\begin{cor}\label{15.06.2015--1}Let $V\in \mathrm{Rep}_R(G)^o$ afford  a faithful representation of $G$. Let $\rho:\mathcal G\to G$ be an arrow of $\mathbf{FGSch}/R$ which is an isomorphism on generic fibres. If $V\ot   R_n$ is the trivial representation of $\mathcal G\ot R_n$  for all $n$, then $\rho$ is the automatic blowup of the identity.
\end{cor}

\begin{proof}Let $[a_{ij}]\in \mathrm{GL}_r(R[G])$ be the matrix associated to $V$ on some unspecified basis. Write $b_{ij}=a_{ij}-\delta_{ij}$. Since $V\ot R_n$ is the trivial representation of $\mathcal G\ot R_n$, we conclude that $\pi^{n+1}|b_{ij}$ in $R[\mathcal G]$ for all $n$. Now, the augmentation ideal  $\g a$ of $R[G]$ is generated by the functions $b_{ij}$, so that $\g a\cdot R[\mathcal G]\subset (\pi^{n+1})$. Hence, $R[\mathcal G]\supset\cup_n R[G][\pi^{-n}\g a]=R[N]$. Using Corollary \ref{06.02.2015--1}, we have $R[N]=R[\mathcal G]$.  
\end{proof}


\section{Faithful representations of Neron blowups}\label{07.10.2014--3} 
Let $G$ be a flat group scheme over $R$. We assume furthermore that $G$ is of finite type over $R$, so that the Neron blowup (see Section \ref{03.02.2015--3}) of any closed subgroup of $G_k$ is again of finite type. As such a group scheme, it admits a closed embedding  into some $\mathbf{GL}_{r,R}$ (adapt the \emph{proofs} in \cite[3.3]{waterhouse}), 
or, according to the terminology at the end of section \ref{05.10.2016--2}, it possesses a faithful representation. (The reader is again warned that this terminology \emph{differs from \cite{SGA3}}, where in Expos\'e I, Definition 2.3.6.2, a faithful action is one having a trivial kernel.) In this section we describe a means to find faithful representations of Neron blowups. Of course, put this way, our task might seem pointless since  it is sufficient to run a general argument. But our point of view is to produce faithful representations by performing linear algebraic constructions in $\mathrm{Rep}_R$. 
This is most desirable when the category $\mathrm{Rep}_R$ is in fact one side of a Tannakian correspondence, see Example \ref{15.10.2014--1}.

The principle is quite simple and we begin with a particular case: 
\[
G'\longrightarrow G
\] 
is the Neron  blowup of the  identity in the closed fibre.  Let $V\in \mathrm{Rep}_R(G)^o$ be a faithful representation of rank $r$; since $G'_k\to G_k$ is the \emph{trivial} morphism, we know that $\mathbf 1^r\ot k\simeq V\ot k$  as  representations of $G'$. We then obtain a diagram 
\[
\xymatrix{ &V\ar[d]\\ \mathbf 1^r\ar[r]_-{\varphi}& V\ot k  }
\]
in the category $\mathrm{Rep}_R(G')$. (Here $\varphi$ is the obvious morphism.)  
Let $V'$ stand for its pull-back. In concrete terms, 
\begin{equation}\label{09.10.2014--1}
V'=\left\{ (\mathbf v,\mathbf e)\in V\oplus \mathbf 1^r\,:\,  \mathbf v\otimes1=\varphi(\mathbf e) \right\}.
\end{equation}

\begin{prop}\label{05.05.2014--2}In the above setting, the representation  $V'$ of $G'$ is faithful.  
\end{prop}

The proof hinges on the following lemma, whose verification is omitted. 

\begin{lem}\label{05.05.2014--1}Let $M$ and $E$ be free $R$-modules. Let $\{\mathbf m_1,\ldots,\mathbf m_r\}$ be a basis for $M$ and     
$\{\mathbf e_1,\ldots,\mathbf e_s\}$ be one for $E$. Let $q\le \min\{r,s\}$ and define $\varphi:E \to M\ot k$ by 
\[
\varphi(\mathbf e_j )=\left\{\begin{array}{ll}\mathbf m_j\ot1,&1\le j\le q\\ \mathbf 0,&\text{otherwise}.
\end{array}\right.
\]
Then the $R$-module 
\[
\left\{(\mathbf m,\mathbf e)\in M\oplus E\,:\, \mathbf m \otimes1=\varphi(\mathbf e) \right\}
\]
is free on the basis 
\[
\begin{array}{lll}
(\pi\mathbf m_1,\mathbf0),&\ldots,&(\pi\mathbf m_r,\mathbf 0)\\  (\mathbf m_1,\mathbf e_1),&\ldots,& (\mathbf m_q,\mathbf e_q)\\ (\mathbf 0,\mathbf e_{q+1}),&\ldots,&(\mathbf0,\mathbf e_s).
\end{array}
\]
\qed
\end{lem}

\begin{proof}[Proof of Proposition \ref{05.05.2014--2}] This is plain linear algebra. Let $\{\mathbf v_1,\ldots,\mathbf v_r\}$ be a basis of $V$ and let $a_{ij}\in R[G]$ stand for the matrix coefficients inducing the representation of $G$. Let $\varepsilon:R[G]\to R$ stand for the counit; by definition of a representation  we have $\varepsilon(a_{ij})=\delta_{ij}$, so that 
\begin{equation}\label{05.05.2014--3}
a_{ij}=\delta_{ij}\cdot 1+b_{ij},
\end{equation}
with $b_{ij}\in\mathrm{Ker}\,\varepsilon$. By  definition of $G'$, there exist $b_{ij}'\in R[G']$ such that 
\begin{equation}\label{05.05.2014--4}
\pi b_{ij}'=b_{ij}.
\end{equation}
Let $\{\mathbf e_1,\ldots,\mathbf e_r\}$ be an ordered basis of $\mathbf 1^r$ which is sent by $\varphi$ to $\{\mathbf v_1\ot1 ,\ldots,\mathbf v_r\ot1\}$.   According to Lemma \ref{05.05.2014--1}, $V'$ is free on  
\[
\{\pi\mathbf v_1,\ldots,\pi\mathbf v_r,\mathbf v_1+\mathbf e_1,\ldots,\mathbf v_r+\mathbf e_r\}.
\]
Employing this basis and the equations \eqref{05.05.2014--3} and \eqref{05.05.2014--4},
we conclude that the matrix defining the representation of $G'$ on $V'$ is  
\[\left[
\begin{array}{c|c}  a_{ij} & b'_{ij} \\\hline O& \delta_{ij} \end{array}
\right].
\]
Since the functions $a_{ij}$ together with $1/\det\,[a_{ij}]$ generate the $R$-algebra $R[G]$, the functions $b_{ij}$ generate the ideal $\mathrm{Ker}\,\varepsilon$; by definition the functions $a_{ij}$, $1/\det\,[a_{ij}]$ and $b_{ij}'=\pi^{-1}b_{ij}$ generate $R[G']$
  and the proof is finished. 
\end{proof}

Before moving to  the search for a faithful representation of a general Neron blowup $G'\to G$, we record some valuable properties of $V'$ and give an example.  

\begin{cor}\label{09.10.2014--2}We maintain the above notations. 
\begin{enumerate}
\item The representation $V'$ is an extension of  $\mathbf1^r$ by $V$. In particular, $V$ is a sub-representation of $V'$. 
\item Let $\xi$ be the evident extension class in $\mathrm{Ext}_{G'}(V\ot k,V)$. Then the class of $V'$ in $\mathrm{Ext}_{G'}(\mathbf 1^r,V)$ is simply the image of $\xi$ under the morphism induced by $\mathbf1^r\to V\ot k$.  
\item The cokernel of the injection $V'\to V\oplus\mathbf1^r$ is annihilated by $\pi$.   
\item The natural arrow $V'\ot K\to (V\oplus\mathbf1^r)\ot K$ is an isomorphism of representations of $G'\ot K$.  
\item Let $\{\mathbf v_1,\ldots,\mathbf v_r\}$  and  $\{\mathbf e_1,\ldots,\mathbf e_r\}$  be respectively ordered bases of $V$ and  $\mathbf 1^r$ as constructed above. Write $\rho_K:G_K\to \mathbf{GL}_{2r,K}$ for the homomorphism associated to the representation $(V \oplus \mathbf 1^r)\otimes K$ by means of these basis. Then,  if  
\[\beta=\left[\begin{array}{ccc|ccc}\pi &&&&&\\ &\ddots& &&\mathrm{Id}_r& \\
&&\pi&&&\\\hline &&&&&\\ &&&&\mathrm{Id}_r&\\ &&&&&\end{array}\right],
\]
it follows that   $G'$ is the closure of the image of $\beta^{-1}\cdot \rho_K \cdot\beta$ in $\mathbf{GL}_{2r,R}$.

\item As representations of $G'$, $V'$ is a sub-representation of $V\oplus\mathbf1^r$, and $V$ is a sub-representation of $V'$. 

\end{enumerate}  
\end{cor}
\begin{proof}The only statements which need  comment are (2) and (5). Concerning (2),  we only need to inform the reader that we follow \cite[p.87]{hilton-stammbach} in constructing $\mathrm{Ext}(V\ot k,V)\to\mathrm{Ext}(\mathbf1^r,V)$.  Concerning (5), we 
 produce the following justifications. Let $[a_{ij}]\in\mathrm{GL}_r(R[G])$ stand for the matrix associated to the representation $V$ and the basis $\{\mathbf v_1,\ldots,\mathbf v_r\}$. We  observe that   
\[\beta^{-1}\cdot \rho_K\cdot\beta=\left[
\begin{array}{c|c}  a_{ij} &(a_{ij}-\delta_{ij})/\pi \\\hline O& \delta_{ij} \end{array}
\right].
\]
Now the above matrix is also the matrix of a faithful representation $G'\to \mathbf{GL}_{2r,R}$, see the proof of Proposition  \ref{05.05.2014--2}. 
Since $G'\ot K\stackrel{\sim}{\to}G\ot K$, the closure of the image of $\beta^{-1}\cdot\rho_K\cdot\beta$ is the same as the closure of the image of $G'\ot K$ by $\beta^{-1}\cdot\rho_K\cdot\beta$. This is $G'$, since two flat and closed subschemes having the same generic fibre must be equal. 
\end{proof}

\begin{ex}Let $G'$ be the Neron blowup of $\mathbf{G}_{m,R}=\mathrm{Spec}\,R[u,1/u]$ at the origin in $\mathbf G_{m,k}$. If $V$ is the obvious representation of $\mathbf{G}_{m,R}$ (corresponding to $\mathrm{id}$), then $V'$ is  the representation of $G'$ corresponding to the matrix 
\[
\begin{bmatrix}\pi&1\\0&1\end{bmatrix}^{-1}\cdot\begin{bmatrix} u&\\&1  \end{bmatrix}\cdot\begin{bmatrix}\pi&1\\0&1\end{bmatrix}=\begin{bmatrix}u&(u-1)/\pi\\0&1\end{bmatrix}.
\]
It follows that $G'$ is 
\[
\left\{ \begin{bmatrix}u&v\\ 0&1\end{bmatrix} \,:\, \pi v+1=u \right\}.
\]
\end{ex}

Let us now assume that 
\[G'\longrightarrow G\]
is the Neron blowup of a closed subgroup scheme $H_0\subset G\ot k$. The idea to construct a faithful representation of $G'$ by means of a faithful representation of $G$ is to express $H_0$ as the stabilizer of some line. This means that the problem is divided into two: 
\begin{enumerate}\item[\textbf{Step 1};] express $H_0$ as a stabilizer,  and \item[\textbf{Step 2};] construct a faithful representation of the Neron blow of a stabilizer in the special fibre. \end{enumerate}
We start by discussing \textbf{Step 2} and presenting our findings as Proposition \ref{09.10.2014--3}. \textbf{Step 1} follows from standard material in the theory of affine group schemes and is explained in Lemma \ref{07.04.2015--1} below.  

Let $V\in\mathrm{Rep}_R(G)^o$ be faithful and let $\mathbf v\ot 1\in V\ot k$ be a nonzero vector. Let $H_0\subset G_k$ stand for the stabilizer of the \emph{line} \[\ell:=\mathbf v\ot k.\] 
Clearly, it is possible to find an ordered basis $\{\mathbf v_1,\ldots,\mathbf v_r\}$
of $V$ such that $\mathbf v_1=\mathbf v$. Let $[a_{ij}]\in\mathrm{GL}_r( R[G])$ be the matrix of coefficients associated to $\{\mathbf v_1,\ldots,\mathbf v_r\}$. 
It then follows that the ideal cutting out $H_0$ in $R[G]$ is simply 
\[(\pi,a_{21},\ldots,a_{r1}).\]
For future use, we write 
\begin{equation}\label{02.06.2015--2}
a_{21}=\pi a_{21}',\ldots,a_{r1}=\pi a_{r1}'
\end{equation}
for functions  $a_{21}',\ldots,a_{r1}'\in R[G']$.

The line $\ell$ is now fixed by $G'_k$ and we obtain a character $G_k'\to\mathbf{G}_{m,k}$ (defined by the group-like element $a_{11}+(\pi)$ in $k[G']$). In order to follow the method developed to treat the particular case, we should find a representation $L\in\mathrm{Rep}_R(G')^o$ lifting $\ell$. 
So here we need to modify the argument since there is no reason for $L$ to exist. 
Let $E\in\mathrm{Rep}_R(G')^o$ be the source of a surjection 
\[
\varphi:E\longrightarrow \ell
\]
in $\mathrm{Rep}_R(G')$. (That $E$ exists is proved in \cite[Proposition 3, p.41]{serre}).  We fix an element $\mathbf e_1\in E$ above $\mathbf v\ot1\in\ell$. A bit of common sense shows that there exists elements $\mathbf e_2,\ldots,\mathbf e_s\in \mathrm{Ker}\,\varphi$ which together with $\mathbf e_1$ form a basis of $E$. 
Let $[b_{ij}]\in\mathrm{GL}_s(R[G'])$ be the matrix associated to the representation $E$ and the ordered basis $\{\mathbf e_1,\ldots,\mathbf e_s\}$. 
Note that in this case we have 
\begin{equation}\label{02.06.2015--1}
\begin{split}b_{11}&=a_{11}+\pi c_{11}\\ b_{12}&=\pi c_{12}\\ \cdots&=\cdots\\ 
b_{1s}&=\pi c_{1s}
\end{split}
\end{equation}
since $\varphi(\mathbf e_1)=\mathbf v_1\otimes1$ and the subspace of $E_k$ generated by $\mathbf e_2\ot1,\ldots,\mathbf e_s\ot1$ is stable under $G_k'$. 
We then consider the pull-back diagram  
\[
\xymatrix{V'\ar[r]\ar[d]\ar@{}[dr]|{\square}&V\ar[d]\\ E\ar[r]_-{\varphi}&V_k,} 
\]
so that $V'$ is naturally a sub-$G'$-module of $V\oplus E$.
From Lemma \ref{05.05.2014--1} we know that 
\[
\begin{array}{lll} (\pi\mathbf v_1,\mathbf0),&\ldots,&(\pi\mathbf v_r,\mathbf 0)\\ (\mathbf v_1,\mathbf e_1),&&\\ (\mathbf 0,\mathbf e_2),&\ldots,&(\mathbf0,\mathbf e_s)
\end{array}
\]
is a basis of $V'$. 
A simple computation using equations \eqref{02.06.2015--2} and \eqref{02.06.2015--1} shows that 
the matrix of the representation of $G'$ on $V'$ associated to this basis is 
\[\left[\begin{array}{ccc|c|ccc}
a_{11}&\cdots&a_{1r}&-c_{11}&-c_{12}&\cdots&-c_{1s}\\
\vdots&\ddots&\vdots&a_{21}'&0&\cdots&0\\ 
\vdots&&\vdots&\vdots&\vdots&\vdots&\vdots\\
a_{r1}&\cdots&a_{rr}&a_{r1}'&0&\cdots&0\\
\hline0&\cdots&0&b_{11}&b_{12}&\cdots&b_{1s}\\
\vdots&\ddots&\vdots&\vdots&\vdots&\ddots&\vdots\\
0&\cdots&0&b_{s1}&b_{s2}&\cdots&b_{ss}
\end{array}\right].
\]
Since the $a_{ij}$ together with $1/\det\,[a_{ij}]$ generate $R[G]$, and $a_{21}',\ldots,a_{r1}'$ generate $R[G']$ over $R[G]$, we conclude that $V'$ is a faithful representation of $G'$. We have then proved the following. 
\begin{prop}\label{09.10.2014--3}Let $V\in\mathrm{Rep}_R(G)$ be faithful and let $\ell\subset V_k$ be a line   whose stabilizer is denoted by $H_0\subset G_k$. Let $G'$ be the Neron blowup of $G$ at $H_0$ and let $\varphi:E\to \ell$ be a surjective morphism in $\mathrm{Rep}_R(G')$ such that the $R$-module $E$ is free. Then $V':=V\times_{V_k}E$ is a faithful representation of $G'$.  \qed
\end{prop}

\begin{cor}Let $V\in\mathrm{Rep}_R(G)^o$ be faithful and let $\mathbf v\ot1\in V_k\setminus0$ be a vector whose stabilizer is denoted by $H_0\subset G_k$. (Note that $H_0$ is smaller than the stabilizer of the line!) Let $G'$ be the Neron blowup of $G$ at $H_0$ and note that there is an obvious $G'$-equivariant morphism  $\mathbf1\to V_k$ whose image is the line  $\mathbf v\ot k$. Then $V':=V\times_{V_k}\mathbf1$ is a faithful representation of $G'$.  Moreover, if $[a_{ij}]\in \mathrm{GL}_r(R[G])$ is the matrix associated to some basis $\{\mathbf v_1,\ldots,\mathbf v_r\}$ with $\mathbf v_1=\mathbf v$, then the matrix of the representation $V'$ of $G'$ is 
\[\left[\begin{array}{ccc|cccc}
a_{11}&\cdots&a_{1r}&\alpha_{11}\\
\vdots&\ddots&\vdots&a_{21}'\\ 
\vdots&&\vdots&\vdots\\
a_{r1}&\cdots&a_{rr}&a_{r1}'\\
\hline0&\cdots&0&1
\end{array}\right],
\]
where $\pi\alpha_{11}=a_{11}-1$ and $\pi a_{i1}'= a_{i1}$ for each $i\ge2$.\qed
\end{cor}

To address \textbf{Step 1} we only need the following. 

\begin{lem}\label{07.04.2015--1}Let $H_0\subset G_k$ be a closed subgroup scheme of $G$. Then there exists  a faithful representation $V$ of $G$ and a line $\ell\subset V_k$ whose stabilizer is exactly $H_0$. 
\end{lem}
\begin{proof}By \cite[16.1, Corollary]{waterhouse} there exists $\rho:G_k\to\mathbf{GL}_{r,k}$ and a line $\ell\subset k^r$ whose stabilizer is $H_0$. Using \cite[Proposition 3, p.41]{serre} we can find $W\in\mathrm{Rep}_R(G)^o$ together with a $G_k$-equivariant injection $k^r\to W_k$. Obviously, $H_0$ is still the stabilizer of the line $\ell\subset W_k$. 
Replacing $W$ by $V=W\oplus F$, with $F$ faithful,    we are done.   
\end{proof}

Let us now criticize Proposition \ref{09.10.2014--3}. Finding explicitly the ``covering'' representation $\varphi:E\to \ell$ is by no means a simple task. (But in many important cases, the character of $G_k'\to \mathbf G_{m,k}$ associated to $\ell$ will be a reduction of a character $G'\to \mathbf G_{m,R}$.) Also, if we abandon the need to construct the faithful representation $V'$ by means of linear algebra in the abelian category $\mathrm{Rep}_R(G')$, a more efficient path is:

\begin{prop}\label{24.06.2015--6}Let $V$ be a free $R$-module on the basis $\mathbf v_1,\ldots,\mathbf v_r$ affording a faithful representation of $G$.
Let $H_0\subset G_k$ be the stabilizer of the line spanned by $\mathbf v_1\ot1$ in $V_k$ and denote by $G'\to G$ the Neron blowup of $H_0$.
Write $[a_{ij}]\in \mathrm{GL}_r(R[G])$  for the matrix associated to the representation of $G$ on $V$, and let    $a_{21}',\ldots,a_{r1}'$ be functions of $R[G']$   which, when multiplied by $\pi$, become respectively  $a_{21},\ldots,a_{r1}$. Then the following claims are true. 

\begin{enumerate}
\item The $R$-submodule of $V_K$ freely generated by $\mathbf v_1\ot\pi^{-1},\mathbf v_2\ot1,\ldots,\mathbf v_r\ot1$ has the structure of an $R[G']$-comodule. Its associated matrix is 
\[\left[
\begin{array}{c|ccc}
a_{11} & \pi a_{12} & \cdots & \pi a_{1r} \\ \hline  a_{21}' & a_{22} & \cdots & a_{2r} \\ \vdots & \vdots &\ddots & \vdots \\
a_{r1}'& a_{r2}&\cdots & a_{rr}  
\end{array}
\right].
\] 
\item If $V'$ stands for the representation of $G'$ considered in the previous item, then $V\oplus V'$ is a faithful representation of $G'$. 
\end{enumerate}
\end{prop}

\begin{proof}Once one knows what to look for, the proof is a triviality. 
\end{proof}

Provided the centre of the Neron blowup has certain particular features, the point of view of \cite{SS} provides another means to construct faithful representations.

\begin{prop}\label{24.06.2015--5}  
Let $H\subset G$ be an $R$-flat, normal subgroup scheme and write $A$ for the quotient group scheme. 
Denote by $G'$ and $A'$ the Neron blowups of $G$ at $H_k$ and of $A$ at $\{e\}\subset A_k$, respectively. Let $G'\to A'$ be the morphism obtained by the ``universal property''   (Lemma \ref{09.04.2015--1}).  If   $\rho:G\to \mathbf{GL}_r$ and  $\sigma:A'\to\mathbf{GL}_s$ are faithful representations,  then $\rho\oplus\sigma$ is a faithful representation of $G'$.  
\end{prop}


\begin{proof} Here are the properties of the quotient which we shall need: $R[A]$ is a subring of $R[G]$ and if $\g a_A$ stands for the augmentation ideal of $R[A]$, then $\g a_A\cdot R[G]$ cuts out $H$. Hence, if   $a_1,\ldots,a_m$ are generators of $\g a_A$,  it follows that \[R[G']=R[G][\pi^{-1}a_1,\ldots,\pi^{-1}a_m].
\]
We  now verify that  
\[
\rho^*\otimes\sigma^*:R[\mathbf{GL}_r]\otimes R[\mathbf{GL}_s] \aro R[G']
\]
is a surjection. By definition of $\rho$, $R[G]\subset \mathrm{Im}(\rho^*\otimes\sigma^*)$. By definition of $\sigma$, the elements $\pi^{-1}a_i$ belong to $\mathrm{Im}(\sigma^*)$. This finishes the proof. 

\end{proof}

\section{Images of morphisms between flat group schemes}\label{06.03.2015--2}
Let $\rho:\Pi\to G$ be a morphism in $\mathbf{FGSch}/R$.  There are two natural ways of defining ``images'' of $\rho$. 
\begin{defn}[The diptych]\label{21.08.2014--1} Define $\Psi_\rho$ as the group scheme whose Hopf algebra is the image of $R[G]$ in $R[\Pi]$. Define $R[\Psi_{\rho}']$ as the saturation of the latter inside $R[\Pi]$. The obvious commutative  diagram 
\[
\xymatrix{   \Psi'_\rho \ar[r]& \Psi_\rho\ar [d] \\  \Pi\ar[r]_\rho \ar [u] & G   }
\]
is called the diptych of $\rho$. 
\end{defn}

Implicit in the above definition is the fact that $R[\Psi_\rho']$ is a Hopf algebra. This can be extracted from the proof of Lemma 3.1.2 in \cite{DH} (see Remarks \ref{21.04.2015--2}).  Another relevant fact, whose proof the reader can find in \cite[Theorem 4.1.1]{DH}, is the following. 

\begin{thm}The morphism $\Pi\to \Psi_\rho'$ is faithfully flat.\qed 
\end{thm}

A basic property of  $\Psi_\rho$ is as follows. 

\begin{lem} The  group scheme $\Psi_\rho$ is  the closure of the image of $\Pi_K  \to G_K$
in $G$.\qed 
\end{lem}
Other useful facts having simple proofs are collected in the next lemma.

\begin{lem}Let $\rho:\Pi\to G$ be as before and consider a factorization of $\rho$ in the category $\mathbf{FGSch}/R$:
\[\Pi\aro G'\aro G.\]
 
\begin{enumerate}
\item If $\Pi\to G'$ faithfully flat, then there exists a unique dotted arrow   
rendering the diagram 
\[\xymatrix{& \Psi_\rho'\ar[dr]& \\ \Pi\ar[ru]\ar[r]&G'\ar[r]\ar@{-->}[u]&G}\] 
commutative. 
\item If $G'\to G$ is a closed embedding, then there exists a unique dotted arrow  
\[\xymatrix{&\ar@{-->}[d]\Psi_\rho\ar[dr]&\\ \Pi\ar[ru]\ar[r]&G'\ar[r]&G}\] 
rendering this diagram commutative. 
\item If $G'\to G$ is a closed embedding,   and the arrow $\Psi_\rho\to G'$ produced in (2) induces an isomorphism on generic fibres, then $\Psi_\rho\to G'$ is an isomorphism.
\item If $\Pi\to G'$ is faithfully flat and the arrow $G'\to \Psi_\rho'$ produced in (1) induces an isomorphism on generic fibres, then $G'\to \Psi_\rho'$ is an isomorphism. 
\qed 
\end{enumerate}
\end{lem}

\begin{lem}\label{09.03.2015--7}If $\Psi'_{\rho,k}\to \Psi_{\rho,k}$ is faithfully flat, then $\Psi_{\rho}'\to \Psi_\rho$ is an isomorphism. 
\end{lem}

\begin{proof}By construction,  $R[\Psi_\rho]\to R[\Psi_\rho']$ is injective.  If $k[\Psi_\rho] \to k[\Psi_\rho']$ is also injective, which follows from the assumption, then $R[\Psi_\rho] \subset R[\Psi_\rho']$ is saturated. The equality $K[\Psi_\rho]=K[\Psi_\rho']$ then finishes the proof.   
\end{proof}

Over the residue field $k$, there is yet another interesting group scheme in sight: the image of $\rho_k$. We then have the \emph{triptych} of $\rho$, which is the  commutative diagram 
\begin{equation}\label{09.05.2014--1}
\xymatrix{  \ar@{->>}[rd]\Psi_{\rho,k}'\ar[rr]&& \Psi_{\rho,k}\ar@{^{(}->}[dd] \\ &\textrm{Im}(\rho_k)\ar@{^{(}->}[ru]\ar@{^{(}->}[rd] & \\ \ar@{->>}[uu]\Pi_k\ar[rr]_{\rho_k}\ar@{->>}[ru]  && G_k .} \qquad\qquad \begin{array}{ll}   \xymatrix{\ar@{^{(}->}[r]&}  =& \text{closed immersion}\\ \xymatrix{\ar@{->>}[r]&}= &\text{ faithfully flat.} \end{array}
\end{equation}
Together with Lemma \ref{09.03.2015--7}, diagram \eqref{09.05.2014--1} proves the following:
\begin{cor}\label{09.05.2014--2}The following claims are true.
\begin{enumerate}\item[i)] If $\mathrm{Im}(\rho_k)\to \Psi_{\rho,k}$ is an isomorphism, then $\Psi'_{\rho}\to \Psi_\rho$ is an isomorphism. 

\item[ii)] If $\Psi_{\rho}'\to \Psi_{\rho}$ is an isomorphism, then $\mathrm{Im}(\rho_k)\to\Psi_{\rho,k}$ is an isomorphism. 

\item[iii)] The image of $\Psi_{\rho,k}'$ in $\Psi_{\rho,k}$ is none other than  $\mathrm{Im}(\rho_k)$.
\end{enumerate}
\qed 
\end{cor}

\begin{prop}\label{24.06.2015--7}The kernel of $\Psi_{\rho,k}'\to \mathrm{Im}(\rho_k)$ is unipotent. 
\end{prop}
\begin{proof}Obviously, the kernel in question is also $\mathrm{Ker}\,\Psi_{\rho,k}'\to \Psi_{\rho,k}$ as we learn from diagram \eqref{09.05.2014--1}. Using Theorem \ref{26.09.2014--6}, we are able to write $\Psi'_\rho$ as $\varprojlim_i\Psi_i'$, where each $\Psi_{i+1}'\to\Psi_{i}'$ is a Neron blowup and $\Psi_0'=\Psi_\rho$. The reader can immediately verify that 
\[
\mathrm{Ker}\left(\Psi'_{\rho,k}\to \Psi_{\rho,k}\right)\quad=\quad\varprojlim_i\mathrm{Ker}\left(\Psi'_{i,k}\to \Psi_{\rho,k}\right).
\]
Consequently, using \cite[Proposition 2.3 on p. 485]{DG}, we only need to show that each kernel in the above   limit is unipotent. This, in turn follows directly from  \cite[Theorem 1.5]{WW} and two fundamental properties of unipotent group schemes: a closed subgroup scheme of a  unipotent group scheme is unipotent, and the extension of a  unipotent group scheme by a unipotent group scheme is still unipotent \cite[Proposition 2.3 on p. 485]{DG}. 
\end{proof}

We show in the following example that it is possible that neither  $\Psi_{\rho,k}'\to \mathrm{Im}(\rho_k)$ nor  $\mathrm{Im}(\rho_k)\to \Psi_{\rho,k}$ be an isomorphism. 

\begin{ex}\label{15.10.2014--2} Let $\rho:G'\to G$ be the Neron blowup of the origin in the closed fibre of $G=\mathbf G_{m,R}$. The diptych of $\rho$ is 
\[
\xymatrix{ G'\ar[r]^{\rho}  & G \\ \ar[u]^{\mathrm{id}}G'\ar[r]_\rho& G\ar[u]_{\mathrm{id}}.  }
\] 
Then $\Psi_{\rho,k}'=\mathbf G_{a,k}$, $\Psi_{\rho,k}=\mathbf G_{m,k}$ and $\mathrm{Im}(\rho_k)=\{e\}$.
\end{ex}

Further ahead, in Example \ref{15.10.2014--1}, we will show how this example fits in a more empiric situation.

Let $V$ be a free $R$-module of finite rank and assume that our $G$ above equals $\mathbf{GL}(V)$. 
We now interpret $\Psi_\rho$ and $\Psi_\rho'$ in terms of their representation categories. This amounts simply to finding proper references in the literature.
\begin{defn}\label{07.04.2015--2} Let $\Pi\in\mathbf{FGSch}/R$, as before. 
\begin{enumerate}
\item Let $V$ be an object of $\mathrm{Rep}_R(\Pi)^o$. Write $\mathbf T^{a,b}V$ for the representation $V^{\ot a}\ot V^{\vee\ot b}$ and denote by $\langle V\rangle_\ot$ the full subcategory of $\mathrm{Rep}_R(\Pi)$ having as objects sub-quotients of direct sums $\mathbf{T}^{a_1,b_1}V\oplus\cdots\oplus\mathbf{T}^{a_m,b_m}V$ for varying $a_i,b_i$. 
\item  Let $\alpha:V'\to V$ be a monomorphism in $\mathrm{Rep}_R(\Pi)$ with both $V$ and $V'$ free as $R$-modules. If $\mathrm{Coker}(\alpha)$ is also free as an $R$-module, we say, following \cite[Definitions 10 and 23]{dS09}, that $\alpha$ is a special monomorphism.  
Call an object $V''\in \mathrm{Rep}_R(\Pi)^o$ a special sub-quotient of $V$ if there exists a special monomorphism $V'\to V$ and an epimorphism $V'\to V''$. The category of all special sub-quotients of various $\mathbf{T}^{a_1,b_1}V\oplus\cdots\oplus\mathbf{T}^{a_m,b_m}V$ is denoted by $\langle V\rangle_\ot^s$.   
\end{enumerate}
\end{defn}

\begin{prop}\label{09.03.2015--2}
Let $V$ be an object of $\mathrm{Rep}_R(\Pi)^o$ and $\rho$ be the natural homomorphism $\Pi\to G:=\mathbf{GL}(V)$.
\begin{enumerate}
\item The obvious functor $\mathrm{Rep}_R(\Psi_\rho)\to\mathrm{Rep}_R(\Pi)$ defines an equivalence of categories between $\mathrm{Rep}_R(\Psi_\rho)^o$  and $\langle V\rangle_\ot^s$.
\item  The obvious functor $\mathrm{Rep}_R(\Psi_\rho')\to\mathrm{Rep}_R(\Pi)$ defines an equivalence between $\mathrm{Rep}_R(\Psi_\rho')$ and $\langle V\rangle_\ot$. 
\end{enumerate}
\end{prop}
\begin{proof}The first claim is a result of \cite[Proposition~12]{dS09}. The second claim follows from the definition of $\Psi_\rho'$ and \cite[Theorem~4.2]{DH}.
\end{proof}

Note that, in general,  $\mathrm{Rep}_R(\Psi_\rho)$ is not a full subcategory of $\mathrm{Rep}_R(\Pi)$.
This means that we have the following interpretation of the diptych (Definition \ref{21.08.2014--1}) in terms of representation categories: 
\[
\xymatrix{\langle V\rangle_\ot\ar[d]& \ar[l]\ar[d]\langle V\rangle^s_\ot\\ \mathrm{Rep}_R(\Pi)&\mathrm{Rep}_R(\mathbf{GL}(V))^o.\ar[l]}
\]

We now deal with the representation theoretic interpretation of the triptych (diagram \eqref{09.05.2014--1}) of $\rho$. For that, given an $R$-linear category $\mathcal C$, write $\mathcal C_{(k)}$ to denote the full subcategory whose objects $W$ are annihilated by $\pi$, i.e. $\pi\cdot\mathrm{id}_W=0$ in $\mathrm{Hom}_{\mathcal C}(W,W)$. We then have a commutative diagram of \emph{solid} arrows between $k$-linear abelian categories: 
\begin{equation}\label{09.03.2015--8}
\xymatrix{\mathrm{Rep}_R(\Psi_\rho')_{(k)}\ar[dd]     && \mathrm{Rep}_R(\Psi_\rho)_{(k)}\ar[ll]\ar@{-->}[dl]  \\ &\ar[lu]\langle V\ot k\rangle_\ot \ar[dl]&\\ \mathrm{Rep}_R(\Pi)_{(k)}.  }
\end{equation}
From \cite[Part I, 10.1, 162]{jantzen} the categories $\mathrm{Rep}_R(-)_{(k)}$ are simply the corresponding representations categories of the group schemes obtained by base change $R\to k$. Since $V$ is a faithful representation of $\Psi_\rho$ (recall that $\Psi_\rho\to\mathbf{GL}(V)$ is a closed immersion by construction), $V\ot k$ is a faithful representation of $\Psi_\rho\ot k$, so that each object of $\mathrm{Rep}_R(\Psi_\rho)_{(k)}$ is a sub-quotient of some $\bigoplus \mathbf T^{a_i,b_i}(V\ot k)$.
 This means that the upper horizontal arrow in diagram \eqref{09.03.2015--8} factors through $\langle V\ot k\rangle_\ot$, i.e. the dotted arrow exists and still produces a commutative diagram. We conclude that diagram \eqref{09.03.2015--8} captures the essence of diagram \eqref{09.05.2014--1} as the former  can easily be completed by introducing the representation category of the general linear group on the lower right corner.


\section{Neron blowups of formal subgroup schemes}\label{05.02.2015--2}

Let $G$ be group scheme over $R$ which is flat and of finite type. In what follows, we fix a non-negative integer $n$. 
Let  $H_n\subset G\ot{R_n}$ be a closed subgroup scheme (over $R_n$) cut out by the ideal $I_n\subset R[G]$. We note that, in this case, $\pi^{n+1}\in I_n$. 

\begin{defn} The subring of $K[G]$ obtained by adjoining to $R[G]$ all elements of the form $\pi^{-n-1}a$ with $a\in I_n$ will be denoted by $E^n$. 
\end{defn}

\begin{lem}\label{23.09.2014--1}
Let $\Delta$, $\varepsilon$, and $S$ denote respectively the co-multiplication, the co-identity and the antipode of $K[G]$. Then $\Delta$, $\varepsilon$ and $S$ send $E^n$ into $E^n\ot E^n$, $R$ and $E^n$ respectively. 
\end{lem}

\begin{proof}
Let $a\in I_{n}$. Since $(\mathrm{Ker}\,\varepsilon,\pi^{n+1})\supset(I_n,\pi^{n+1})$, there exists some $c\in \mathrm{Ker}\,\varepsilon$ and some $f\in R[G]$ such that $a=c+\pi^{n+1}f$. It then follows that $\varepsilon(a)=\pi^{n+1}\varepsilon(f)$, so that $\varepsilon(a\pi^{-n-1})\in R$. Also, there are $a_i,a_i'\in 
I_n$ together with $\tau\in R[G]\ot R[G]$ such that 
\[
\Delta(a)=\sum a_i\ot x_i+x_i'\ot a_i'+\pi^{n+1}\tau. 
\]
Then, 
\[
\Delta\left(\pi^{-n-1}a\right)= \sum \pi^{-n-1}a_i\ot x_i+x_i'\ot\pi^{-n-1} a_i'+\tau.
\]
The verification of the statement concerning the antipode is equally trivial and is omitted. 
\end{proof} 

\begin{defn}\label{20.11.2014--1}The group scheme  
\[\mathrm{Spec}\,E^n\]
is called the Neron blowup of $H_n$.  
\end{defn}

Obviously,  there is a morphism of group schemes \[\mathrm{Spec}\,E^n\aro G\] which, when tensored with $K$, becomes an isomorphism. The next result, whose simple proof is left to  the reader, says a bit more about the group scheme $\mathrm{Spec}\,E^n$.
\begin{lem}\label{01.12.2016--1}The following claims are true. 
\begin{enumerate}
\item The morphism $\mathrm{Spec}\,E^n\to G$ when tensored with $R_n$ factors thorough $H_n\subset G\otimes R_n$. 
\item If $f:G'\to G$ is an arrow of $\mathbf{FGSch}/R$ such that $f\ot R_n$ factors through $H_n\subset G\ot R_n$, then $f$ factors uniquely through $\mathrm{Spec}\,E^n\to G$. \qed
\end{enumerate}
\end{lem}

We now wish to concentrate on the case ``$n= \infty$.'' We assume that $R$ is \emph{complete}.  Let $\widehat G$ be the   completion of $G$ along its closed fibre (it is automatically flat over $R$) and let  $\mathfrak H\subset\widehat G$ be a closed formal subgroup scheme which is moreover flat over $R$.  
We follow traditional notation in the theory of adic algebras and write $R\langle  G\rangle$, respectively $R\langle \mathfrak H\rangle$, to denote the algebra associated to the formal schemes $\widehat G$, respectively $\mathfrak H$. This being so, $R\langle \mathfrak H\rangle$ is a quotient of $R\langle  G\rangle$ by some ideal $I$. 
In order to carry on proofs, we shall find useful to let $I_n\subset R[G]$ be the ideal of the closed subgroup $H_n$ of $G\ot{R_n}$ induced by $\mathfrak H$. (The integer $n$ should no more be thought as fixed.) Note that $\pi^{n+1}\in I_n$ and $I_n R\langle  G\rangle=(\pi^{n+1},I)$. 
Using Lemma \ref{01.12.2016--1}, we see that $E^n\subset E^{n+1}$.

\begin{defn}\label{25.09.2014--2}We denote by $E^\infty$ the algebra $\cup_n E^n$. 
\end{defn}

Of course, replacing ``$n$'' by ``$\infty$'' in the statement of Lemma \ref{23.09.2014--1} still produces a true claim so that we have the following. 

\begin{defn}\label{25.09.2014--1}The group scheme associated to the $R$-algebra  $E^\infty$ of Definition \ref{25.09.2014--2} will be denoted by $\cn^\infty_{\g H}$ or $\cn^\infty_{\g H}(G)$ if the need presents. The group scheme $\cn^\infty_{\g H}$ will be called the blowup of $G$ along   $\g H$. 

The group scheme $\mathrm{Spec}\,E^n$ associated to $H_n=\g H\ot R_n$ (see Definition \ref{20.11.2014--1}) will be denoted by $\cn^n_{\g H}$ or $\cn^n_{\g H}(G)$ if the need presents. The group scheme $\cn^n_{\g H}$ will be called the partial blowup of $G$ of level $n$ along   $\g H$. 

If $\g H$ is the completion of a closed subgroup $H$ of $G$, we put $\cn_H^*=\cn_{\g H}^*$. 
\end{defn}

\begin{ex}The automatic blowup of the identity introduced in Definition \ref{22.08.2014--2} is, due to Proposition \ref{05.02.2015--1},  $\cn_{\{e\}}^\infty(G)$.
\end{ex}

Note that the obvious arrow of groups $\cn_{\g H}^\infty\to G$ becomes an isomorphism when tensored with $K$. 
The following is a trivial observation which will prove useful further ahead.  

\begin{lem}\label{26.09.2014--1}Let $\varphi:G'\to G$ be a morphism in $\mathbf{FGSch}/R$. Denote by $\hat\varphi:\widehat{G}'\to\widehat{G}$ the morphism induced between $\pi$-adic completions. If $\g H'\subset \widehat{G}'$ is a closed and $R$-flat formal subgroup which is taken by $\hat\varphi$ into $\g H$, then there exists a unique arrow of affine group schemes $\cn^\infty_{\g H'}(G')\to \cn^\infty_{\g H}(G)$ rendering the diagram 
\[
\xymatrix{ \ar[d] \cn^\infty_{\g H'}(G')\ar[r]&G'\ar[d]^\varphi& \\  \cn^\infty_{\g H}(G)\ar[r]&G  }
\]   
commutative. 
\qed
\end{lem}

In the following lines we wish to present some elementary features of the blowup along a formal subgroup. We maintain the above notations and introduce 
\[K\langle G\rangle=K\otimes R\langle G\rangle\quad\text{and}\quad K\langle\g H\rangle=K\ot R\langle \g H\rangle.\] 
For what comes, the following commutative  diagram can be helpful.
\[
\xymatrix{ \ar@{^{(}->}[d] R[G]\ar[r]&\ar@{^{(}->}[d] R\langle  G\rangle \ar@{->>}[r] & R\langle\mathfrak  H\rangle \ar@{^{(}->}[d]\\ K[G]\ar[r]&\ar[r] K\langle  G\rangle\ar@{->>}[r]  & K\langle \mathfrak H\rangle.  }
\]

\begin{lem}An element $b\in K[G]$ belongs to $E^\infty$ if and only if its image in $K\langle \g H\rangle$ belongs to the image of $R\langle\g H\rangle$. 
\end{lem}

\begin{proof}Let $b\in R[G]$ and  $m\ge1$. We assume that the image of $\pi^{-m}b$ in $K\langle\g H\rangle$ coincides with the image of an element  $a^*\in R\langle   G\rangle$. Then $\pi^m a^*$ and $b$ have the same image in $K\langle\g H\rangle$. This  implies that $\pi^m a^*\equiv b\mod I$.   As $I R_n[G]=I_n R_n[G]$, we have $0\equiv b\mod I_{m-1} R_{m-1}[G]$. Hence $\pi^{-m}b\in E^{m-1}(G)$.

Let $a\in I_n$. Since $(I,\pi^{n+1})=I_nR\langle G\rangle$, we can write $a=\pi^{n+1}a^*+a^{**}$, where $a^*\in R\langle G\rangle$ and $a^{**}\in I$. Then, $\pi^{-n-1}a=a^*+\pi^{-n-1}a^{**}$ in $K\langle G\rangle$ with $a^*\in R\langle G\rangle$ and $\pi^{-n-1}a^{**}\in I\cdot K\langle  G\rangle$. The image of $\pi^{-n-1}a$ in $K\langle\g H\rangle$ is then the image of $a^*$. 
\end{proof}

The next statement employs the notion of fibre product of rings, cf. \cite[Section 1]{ferrand}.

\begin{cor}\label{06.06.2015--1}The first projection \[K[G]\times_{K\langle\mathfrak H\rangle}R\langle\mathfrak H\rangle\longrightarrow K[G]\]
induces an isomorphism between $K[G]\times_{K\langle\mathfrak H\rangle}R\langle\mathfrak H\rangle$ and $E^\infty$.\qed
\end{cor}

The next result marks a distinctive feature of the Neron blowup of a formal  flat subgroup as opposed to Neron blowups centered at the special fibre (cf. the second remark after Theorem 1.4 in \cite{WW}). 

\begin{cor}\label{24.09.2014--2}For each $n\in\mathbb N$, the obvious morphism of $R_n$-groups  $\cn^\infty\otimes R_n\stackrel{}{\to}G\otimes R_n$ induces an isomorphism of $\cn^\infty\otimes R_n$ with $H_n=\mathfrak H\otimes R_n$. 
\end{cor}
\begin{proof} In view of Corollary \ref{06.06.2015--1}, it is enough to show that the obvious morphism 
\[\theta:R[G]\aro K[G]\times_{K\langle\mathfrak H\rangle}R\langle\mathfrak H\rangle,
\]
when tensored with $R_n$, gives the ensuing factorization: 
\[\xymatrix{
R_n[G] \ar[d]\ar[r]& R_n\ot\left( K[G]\times_{K\langle\mathfrak H\rangle}R\langle\mathfrak H\rangle \right) \\ R_n\ot R\langle\g H\rangle\ar[ru]_\sim& 
}\]

We consider the  commutative diagram of $R$-algebras
\[\xymatrix{K[G]\times_{K\langle\mathfrak H\rangle}R\langle\mathfrak H\rangle\ar[r]^-{\mathrm{pr_2}}& R\langle\g H\rangle \\ R[G]\ar[u]^\theta\ar[r]& \ar@{->>}[u]R\langle G\rangle,}\]
and claim that $\mathrm{pr}_2\ot R_n$ is an isomorphism. Surjectivity is easily checked  and  injectivity is justified as follows: If $(f,\varphi)\in K[G]\times_{K\langle\mathfrak H\rangle}R\langle\mathfrak H\rangle$ is such that $\varphi=\pi^{n+1}\varphi'$, then $(f,\varphi)=\pi^{n+1}(\pi^{-n-1}f,\varphi')$, so that $(f,\varphi)$ belongs to $(\pi^{n+1})$. 
 The proof then follows by tensoring the above commutative square with $R_n$. 
\end{proof}

We now analyze the standard sequence (see Definition \ref{19.11.2014--1}) associated to $\cn_{\g H}^\infty\to G$. 
We will see that this sequence can be produced by a spontaneous process. For that we need the following concept. 

\begin{defn}[Strict transform]\label{19.11.2014--2}Let $\mathfrak H\subset\widehat G$ be as before.  
Write  $I\subset R\langle G\rangle$ for the ideal of $\g H$ and let $\rho:G'\to G$ be the Neron blowup of some closed subgroup of $G\ot k$.  Then the  strict transform of $\mathfrak H$, denoted $\rho^\#(\mathfrak H)$, is the closed   formal subgroup of $\widehat G'$ cut out by the saturation of the ideal $I\cdot R\langle G'\rangle$, namely  $\cup_n\left(I\cdot R\langle G'\rangle:\pi^n\right)$. If $H\subset G$ is a closed $R$-flat subgroup, we can equally define the strict transform $\rho^\#(H)$, which now is a closed, subgroup of $G'$.   
\end{defn}

\begin{rmks}\begin{enumerate}\label{17.01.2017--1}
\item By construction, all strict transforms are flat over $R$.
\item The strict transform of closed subgroups $H\subset G$ is used in \cite[Section 1]{SS}. 
\item If $\rho:G'\to G$  and $H\subset G$ are as in Definition \ref{19.11.2014--2}, then the strict transform $\rho^\#(H)$ is just the closure in $G'$ of the closed subgroup $H\ot K$ in $G'\ot K$. The same reasoning can be applied in the case of a formal subgroup as $\mathfrak H\subset \widehat G$ if we use the concept of schematic closure in the setting of rigid geometry, see \cite[0.2.4(vi), p.17]{berthelot91}. (Note that   $K\langle G\rangle\to K\langle G'\rangle$ is not usually an isomorphism.) 
\item The completion $\widehat A$ of a noetherian $R$-algebra $A$ is flat an an $A$-module \cite[Theorem 8.8, p.60]{matsumura}. From this, we conclude that for any noetherian $R$-algebra $A$ and for any ideal $J$ of $A$, the equality  $J^\mathrm{sat}\widehat A=(J\widehat A)^\mathrm{sat}$ holds. This proves that, in the notation of the previous item,  the completion of $\rho^\#(H)$ is $\rho^\#(\widehat H)$.
\end{enumerate}
\end{rmks}

We are now ready to state:  
\begin{thm}\label{26.09.2014--7}Let $\mathfrak H\subset \widehat G$ be as above. 
\begin{enumerate}
\item Let $G_0:=G$,  $\g H_0:=\g H$ and denote by  $\rho_0:G_1\to G_0$  the Neron blowup of $\g H_0\ot k$. If $\g H_n$ is defined, write $\rho_n:G_{n+1}\to G_n$ for the Neron blowup of $\g H_n\ot k$ and put $\g H_{n+1}:=\rho_n^\#(\g H_n)$. Then 
\[
\cdots\stackrel{\rho_1}{\longrightarrow} G_1\stackrel{\rho_0}{\longrightarrow}G_0=G\]  
is the standard sequence (see Definition \ref{19.11.2014--1}) of $\cn^\infty_{\g H}(G)\to G$.
\item The image of the canonical morphism $\mathcal N_{\g H}^\infty(G)\ot k\to G_{n+1}\ot k$ is taken isomorphically by $\rho_n\ot k$ onto the image of $\mathcal N_{\g H}^\infty(G)\ot k\to G_{n}\ot k$.
\end{enumerate} 
\end{thm}

The proof will need the following preparation. 
From Corollary \ref{24.09.2014--2} we know that the image of $\rho_0\ot k$ is just $\g H_0\ot k$. Let $a_1^*,\ldots,a_r^*$ generate the ideal  $I$ of $\g H$ and let $a_{i}\in R[G_0]$ be such that $a_{i}^*\equiv a_{i}\mod\pi$. 
Obviously, the ideal of $\g H_0\ot k$ is $(\pi,a_{1},\ldots,a_{r})$. The Neron blowup $\rho_0:G_1\to G_0$ of $\g H_0\ot k$ is given by the inclusion  
\[
R[G_0]\subset E=R[G_0]\left[\frac{a_1}{\pi},\ldots,\frac{a_r}{\pi}\right]. 
\]

In the what follows, we write 
  $A\langle \xi_1,\ldots,\xi_r\rangle$ for the $\pi$-adic completion of $A[ \xi_1,\ldots,\xi_r]$. 

\begin{lem}\label{25.09.2014--3}Let $\widehat E$ stand for the $\pi$-adic completion of $E$. Then 
\[\widehat E\simeq \frac{R\langle G_0\rangle\langle\zeta_1,\ldots,\zeta_r\rangle}{(\pi\zeta_1-a_1^*,\ldots,\pi\zeta_r-a_r^*)^\mathrm{sat}}.\]
\end{lem}

\begin{proof}We use bold letters to stand for  $r$-tuples.
It is clear that, as $R\langle G_0\rangle$-algebras,  
\[\frac{R\langle  G_0\rangle\langle\boldsymbol \xi\rangle}{(\pi\boldsymbol \xi-\boldsymbol a)}\simeq \frac{R\langle  G_0\rangle\langle\boldsymbol \zeta\rangle}{(\pi\boldsymbol  \zeta-\boldsymbol a^*)}
\]
by means of $\xi_i\mapsto\zeta_i-\pi^{-1}(a_i^*-a_{i})$. 
We then recall that for a noetherian $R$-algebra $A$, its completion $\widehat A$ is $A$-flat \cite[Theorem 8.8, p.60]{matsumura}. From this, we conclude that for any noetherian $R$-algebra $A$ and for any ideal $J$ of $A$, the equality  $J^\mathrm{sat}\widehat A=(J\widehat A)^\mathrm{sat}$ holds. Now
\begin{align*}\widehat E&= \left(\frac{R[G_0][\boldsymbol\xi]}{(\pi\boldsymbol\xi-\boldsymbol a)^\mathrm{sat}}\right)^\wedge  & \text{(by Remark \ref{26.09.2014--4})}\\ 
& = \frac{R\langle G_0\rangle\langle\boldsymbol\xi\rangle}{(\pi\boldsymbol\xi-\boldsymbol a)^\mathrm{sat}R\langle G_0\rangle\langle\boldsymbol\xi\rangle} & \text{(by \cite[Theorem 8.7,p.60]{matsumura})}\\ &= \frac{R\langle G_0\rangle\langle\boldsymbol\xi\rangle}{\left(\pi\boldsymbol\xi-\boldsymbol a\right)^\mathrm{sat}}&\text{(by the argument above).}
\end{align*}
\end{proof}

The algebra $\widehat E$ in the statement of the lemma is quite close to a fundamental object in rigid analytic Geometry. Following the terminology of  \cite[4.1]{vdpf}, the analytic space $\mathrm{Sp}\,K\ot \widehat E$ is just the \emph{rational domain} $|a^*_i|\le|\pi|$ in the rigid analytic space $\mathrm{Sp}\,K\langle  G_0\rangle$.

\begin{prop}\label{26.09.2014--2} The morphism of affine formal schemes $\widehat G_1\to \widehat G_0$ induces an isomorphism between $\rho_0^\#(\g H_0)$ and $\g H_0$. The same statement also holds if, instead of a formal subgroup $\g H_0\subset \widehat G_0$, we choose a closed subgroupo $H_0\subset G_0$ which is flat over $R$. 
\end{prop}
\begin{rmk}In the usual theory of blowups, this is known as the statement that the strict transform \emph{is} the blowup of the intersection \cite[Proposition IV-21]{eisenbud_harris} and that blowing up a ``hypersurface'' operates no change. In the present setting, a separate consideration has to be made due to the fact that the generic fibre of $\widehat G_1$ is not the generic fibre of $\widehat G_0$. 
\end{rmk}
\begin{proof}We will omit the verification of the statement concerning a closed subgroup $H_0\subset G_0$ and concentrate on the formal case. 

Write $I^\#$ for the ideal of $\rho_0^\#(\g H_0)$.  We must show that the obvious arrow 
$\phi:R\langle G_0\rangle/I\to R\langle G_1\rangle/I^\#$
is an isomorphism. 
By the description of $\widehat E=R\langle G_1\rangle$ offered in Lemma \ref{25.09.2014--3}, the equality $I=(a_1^*,\ldots,a_r^*)$ and the definition of $I^\#$ as the saturation of $I\cdot R\langle G_1\rangle$, it is evident that $\phi$ is surjective. Injectivity will follow from injectivity of $\phi\ot K$, which for those accustomed to rigid analytic geometry is a triviality. Indeed, there exists a dotted arrow rendering the diagram 
\[\xymatrix{  K\langle G_1\rangle \ar@{-->}[dr] &\\ K\langle  G_0\rangle\ar[r]\ar[u]  &   K\langle G_0\rangle/I\cdot K\langle G_0\rangle   }\]
commutative, see \cite[4.1.2, p.71]{vdpf}. 
The injectivity of $\phi\ot K$ is then obvious as $I\cdot K\langle  G_1\rangle=I^\#\cdot K\langle  G_1\rangle$. 
\end{proof}

\begin{proof}[Proof of Theorem \ref{26.09.2014--7}.]
We have now obtained an affine and flat group scheme $G_1$ together with a closed formal subgroup scheme $\rho_0^\#(\g H_0)=\g H_1\subset \widehat G_1$. By functoriality (Lemma \ref{26.09.2014--1}), we arrive at a commutative diagram 
\[
\xymatrix{\ar[d] \cn^\infty_{\g H_1}(G_1)\ar[r] &\ar[d] G_1\\ \cn^\infty_{\g H_0}(G_0)\ar[r] & G_0.   }
\]
Furthermore, inserting the obvious morphism $\cn^\infty_{\g H_0}(G_0)\to G_1$ in the above diagram still produces a commutative one.   (Recall that $G_1$ is the Neron blowup of $\g H_0\ot k$.)
Using Corollary \ref{24.09.2014--2} and Corollary \ref{26.09.2014--2} we conclude that 
\[
\cn^\infty_{\g H_1}(G_1)\ot k\stackrel{\sim}{\longrightarrow}\cn^\infty_{\g H_0}(G_0)\ot k. \]
As 
\[
\cn^\infty_{\g H_1}(G_1)\ot K\stackrel{\sim}{\longrightarrow}\cn^\infty_{\g H_0}(G_0)\ot K,
\]
we deduce that $\cn_{\g H_1}^\infty(G_1)\to \cn_{\g H_0}^\infty(G_0)$ is an isomorphism  and the image of $\cn^\infty_{\g H_0}(G_0)\ot k\to G_1\ot k$ is the image of $\cn^\infty_{\g H_1}(G_1)\ot k\to G_1\ot k$.  Theorem \ref{26.09.2014--7} can now be proved by induction. 
\end{proof}

To end this section, we use Theorem \ref{26.09.2014--7} to study the standard sequence of the partial blowup of level $m$ of $\g H\subset\widehat G$ (cf. Definition \ref{20.11.2014--1}).   One probable candidate for this sequence is the truncation of the standard sequence of $\cn_{\g H}^\infty(G)$. The situation here seems to be a bit more subtle and we only propose an easy consequence of Theorem \ref{26.09.2014--7}.

\begin{cor}\label{28.10.2014--2}We maintain the notations of Theorem  \ref{26.09.2014--7}.  Let $n\ge0$ be given. Then there exists $\mu\ge0$ depending on $n$ and $\g H$ such that the standard sequence of $\cn_{\g H}^m\to G$, \emph{for all} $m\ge\mu$ starts as 
\[
G_{n+1}\stackrel{\rho_n}{\longrightarrow} \cdots \stackrel{\rho_0}{\longrightarrow} G_0=G.
\]
\end{cor}

\begin{proof}We write $J_\nu\subset R[G_\nu]$ for the ideal of $\g H_\nu\ot k\subset \widehat G_\nu\ot k=G_\nu\ot k$. 
From Theorem \ref{26.09.2014--7}, we know that \begin{enumerate}\item each $R[G_\nu]$ is contained in $E^\infty=\cup_m E^m$ and \item the ideal $\pi  E^\infty\cap R[G_\nu]$, which cuts out the image of $\cn_{\g H}^\infty(G)\ot k$ in $G_\nu\ot k$, is just $J_\nu$.\end{enumerate}
Since $G=G_0$ is assumed to be of finite type, there exists some $\mu_0\ge0$ such that $R[G_{n+1}]\subset E^m$ for all $m\ge\mu_0$. We now take $\mu\ge \mu_0$ to be such that the generators of each one of the ideals $J_0,\ldots,J_n$ belong   to $\pi E^\mu$. This $\mu$ is the one predicted in the statement as in this case, for $m\ge \mu$,  the image of $\cn_{\g H}^m(G)\ot k\to G_\nu\ot k$ is cut out by $J_\nu$ for each $\nu\le n$. 
\end{proof}

\begin{question}Under which conditions does the standard sequence of the ``partial blowup'' $\cn_{\g H}^n(G)\to G$  (see Definition \ref{20.11.2014--1}) begin as $G_{n+1}\to\cdots\to G_0$, where the $G_i$ are as in Theorem \ref{26.09.2014--7}?
\end{question}

\section{Study of a particular class of standard sequences}\label{20.08.2015--1}

In this section, we assume that $R$ is complete. Let $G$ be a group scheme over $R$ which is flat and of finite type.  
Theorem \ref{26.09.2014--7} guarantees that the centres appearing in the standard sequence of 
\[\cn_{\g H}^\infty(G)\aro G\] 
are all isomorphic.  
In this section we wish to understand a possible converse for this: Corollary \ref{09.04.2015--3}.

As the next paragraph argues, standard sequences with ``constant'' centres may easily appear;  the utility of this study is therefore not   restricted to proving a converse to Theorem \ref{26.09.2014--7}. 

Assume that $k$ is of characteristic zero and let \[\cdots\longrightarrow G_n\longrightarrow G_{n-1}\longrightarrow\cdots\longrightarrow G_0\] be a standard sequence (cf. Definition \ref{03.02.2015--2}).
As  $n\mapsto\dim B_n$ is  nondecreasing and  
\[\dim B_n\le\dim G_{n,k}=\dim G_{0,K}\] (the equality follows from [SGA3, $\mathrm{VI}_B$, Corollary 4.3]) there exists $n_0\in\mathbf N$ such that $\dim B_{n_0}=\dim B_n$ for all $n\ge n_0$.
Since the kernel of $B_{n+1}\to B_n$
is either trivial or positive dimensional (it is a subgroup scheme of some $\mathbf G_{a,k}^r$ due to Theorem \ref{11.02.2015--1}) we conclude that for $n\ge n_0$, the arrows $B_{n+1}\to B_n$ are all \emph{isomorphisms}.

It is quite possible that a result more general than Corollary \ref{09.04.2015--3} holds, but for the moment, our best effort needs higher control on  the relation between centres (i.e., the conclusion of Proposition \ref{17.11.2014--1} hold) so that more hypothesis were introduced.

\begin{prop}\label{17.11.2014--1}
Let 
\[
\cdots\stackrel{}{\longrightarrow}G_{n+1}\stackrel{\rho_n}{\longrightarrow}\cdots\stackrel{\rho_0}{\longrightarrow}G_0=G
\]
be a standard sequence (Definition \ref{03.02.2015--2}) where the centre of $\rho_n$ is $B_n\subset G_n\ot k$. Assume the following particularities. 
\begin{enumerate}
\item[i)] The group $G_0$ is smooth over $R$, and   $B_0$ is smooth over $k$. 
\item[ii)] For every $n\ge0$, the restriction  $\rho_{n}\ot k:B_{n+1}\to B_n$ is an isomorphism. 
\item[iii)] There exists an $R$-flat closed subgroup $L_0\subset G_0$ with $L_0\ot k=B_0$.  
\item[iv)] For each representation $V$ of $B_0$, the first cohomology group $H^1(B_0;V)$ vanishes (linear reductivity).  
\end{enumerate}

Then, for every $n\ge1$, there exist  an inner automorphism $a_n:G_n\to G_n$ together with a  closed and $R$-flat subgroup $L_n\subset G_n$ enjoying the following properties. 

\begin{enumerate}\item For any $n\ge1$, the closed subgroup $L_n$ is  $a_n(\rho_{n-1}^\#L_{n-1})$. (See Definition \ref{19.11.2014--2} for notation.) 

\item The special fibre of $L_n$ is $B_n$. 
\end{enumerate}
\end{prop}

\begin{proof}We construct $a_1$ and $L_1$.  From assumption (ii) and \cite[Theorem 14.1]{waterhouse}, we know that $G_1\ot k\to B_0$ is faithfully flat; we arrive at an exact sequence 
\[\tag{e}1\longrightarrow
U_1\longrightarrow G_{1}\ot k\longrightarrow B_0\longrightarrow1
\]
which, again by property (ii), ensures that $G_1\ot k$ is a semidirect product $U_1\rtimes B_1$. In $G_1\ot k$, besides $B_1$, we have another closed subgroup isomorphic to $B_0$, viz. the special fibre of $\Lambda_1:=\rho_0^\#(L_0)$ (see Definition \ref{19.11.2014--2} and Proposition \ref{26.09.2014--2}). As $U_1$ is isomorphic to $\GG_{a,k}^r$, due to (i) and  Theorem \ref{11.02.2015--1}, the action of $B_0$ on $U_1$ by conjugation defines a  linear action of $B_0$ on $U_0$. (See Section \ref{12.08.2015--10} for details.) Assumption (iv) together with the exercise proposed in I.5.1 of \cite{galcoh} shows that there exists an inner automorphism $a_1:G_1\ot k\to G_1\ot k$ taking $\Lambda_1\ot k$ to $B_1$. 

Now $G_1$ is smooth over $R$ due to the ``fibre-by-fibre'' smoothness criterion [EGA $\mathrm{IV}_4$, 17.8.2, p. 79], the isomorphism  $G_1\ot K\simeq G_0\ot K$ and smoothness of  $G_1\ot k\to B_0$ (Theorem \ref{11.02.2015--1}).  
The infinitesimal lifting criterion allows us to find an inner automorphism of $G_1$, parsimoniously called $a_1$, such that $a_1(\Lambda_1\ot k)=B_1$; define $L_1:=a_1(\Lambda_1)$.
This deals with the case $n=1$. 
The inductive step ``$n$ $\Rightarrow$ $n+1$'' proceeds in the exactly same fashion. 
\end{proof}

It is perhaps worth expressing the conclusion of Proposition \ref{17.11.2014--1} in a pictorial form by means of the following \emph{stair}:  
\begin{equation}\label{stair}
\xymatrix{&&\text{higher level}\ar@{}[d]|{\vdots}\\ &\ar[d]^{\rho_1}G_2\ar[r]_{a_2}& G_2\\ G_1\ar[d]^{\rho_0}\ar[r]^{a_1}& G_1& \\ G_0&&}
\end{equation}
where 
\begin{equation}\label{stair2}
\begin{array}{lll} \rho_0&=& \text{blowup of $L_0\ot k$},\\ a_1&=&\text{inner automorphism of $G_1$},  \\    L_1&=&a_1\left[\rho_0^\#(L_0)\right],\\
\rho_1&=&\text{blowup of $L_1\ot k$}, \\ a_2&=&\text{inner automorphism of $G_2$}\\ L_2&=&a_2\left[\rho_1^\#(L_1)\right],\\\cdots&=&\cdots. \end{array}
\end{equation}

For the sake of discussion, we make some definitions. 
\begin{defn}[Spontaneous sequences]\label{18.11.2014--3}
A standard sequence as depicted in the above stair is called an almost spontaneous  sequence associated to $L_0\subset G_0$. 
 An almost spontaneous standard sequence is spontaneous if all the inner automorphisms appearing in it  equal the identity. 
\end{defn}

Using Remark \ref{17.01.2017--1}-(4),   another way of stating Theorem \ref{26.09.2014--7}(1) is then: 
\begin{thm}\label{24.11.2014--5} Let $H\subset G$ be a closed immersion in $\mathbf{FGSch}/R$.  Then the standard sequence of $\cn^\infty_{H}(G)\to G$ is the spontaneous sequence  of $H\subset G$. \qed
\end{thm}

We now wish to prove Theorem \ref{24.11.2014--3} below which sheds light on the nature of \emph{almost spontaneous sequences}. 
\label{24.11.2014--4} We fix, in addition to $G$, a closed immersion $L_0\subset G_0=G$ in $\mathbf{FGSch}/R$ and a stair as in \eqref{stair}-\eqref{stair2}, which is the almost spontaneous sequence associated to $L_0$ and the inner automorphisms  $a_i$. 

\begin{prop}\label{24.10.2014--1}We maintain the above notations. For each $n\ge1$, define inner automorphisms $a_{0,n},\ldots,a_{n,n}$ by decreeing that $a_{n,n}=a_n$ and that 
\[
\xymatrix{G_m\ar[rr]^-{a_{m,n}}\ar[d]_{\rho_{m-1}} && G_m\ar[d]^{\rho_{m-1}} \\ G_{m-1}\ar[rr]_-{a_{m-1,n}}&& G_{m-1}  } 
\]
commutes. 
In other words, we complete the stair \eqref{stair} as suggested by:
\[\tag{case $n=2$}
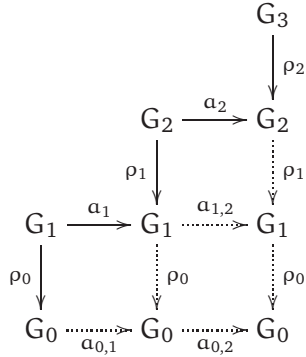
\xymatrix{   &&G_3 \ar[d]^{\rho_2}\\ 
&G_2\ar[r]^{a_2}\ar[d]_{\rho_1} & G_2\ar@{.>}[d]^{\rho_1}\\
G_1\ar[r]^{a_1}\ar[d]_{\rho_0} & G_1\ar@{.>}[r]^{a_{1,2}}\ar@{.>}[d]^{\rho_0}& G_1\ar@{.>}[d]^{\rho_0}\\ 
G_0\ar@{.>}[r]_{a_{0,1}}& G_0\ar@{.>}[r]_{a_{0,2}}& G_0     }\]
Then 
\[
G_{n+1}\stackrel{\rho_n}{\longrightarrow}\cdots \stackrel{\rho_0}{\longrightarrow}G_0=G
\] 
is the spontaneous sequence associated to  
\[
a_{0,n}\cdots a_{0,2} a_{0,1}(L_0)
\]
and truncated at level $n+1$. 
\end{prop}

\begin{proof}We proceed by induction and begin with $n=1$.  In this case, we have a commutative diagram 
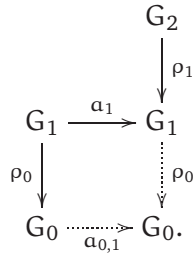
\[
\xymatrix{&G_2\ar[d]^{\rho_1}\\ G_1\ar[r]^{a_1}\ar[d]_{\rho_0}&G_1\ar@{.>}[d]^{\rho_0}\\ G_0\ar@{.>}[r]_{a_{0,1}}&G_0.}
\]
We remark  that  $a_{0,1}:G_0\ot k\to G_0\ot k$ is conjugation by an element of $L_0\ot k$. Hence, $L_0\ot k=a_{0,1}(L_0)\ot k$ and commutativity of the above diagram guarantees that $\rho_0^\#\left[a_{0,1}(L_0)\right]=a_1\left[\rho_0^\#(L_0)\right]$. The desired result then follows.

Let $n>1$ and assume that the result is valid for all almost  spontaneous sequences associated to $L_1=a_1\left[\rho_0^\#(L_0)\right]\subset G_1$  and truncated at level $n$. 
Since 
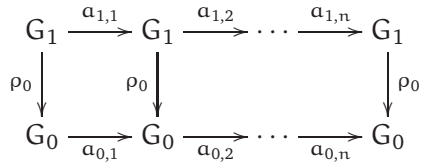
\[
\xymatrix{G_1\ar[r]^{a_{1,1}} \ar[d]_{\rho_0} & G_1\ar[r]^{a_{1,2}}\ar[d]_{\rho_0}& \cdots \ar[r]^{a_{1,n}}& G_1\ar[d]^{\rho_0} \\ G_0\ar[r]_{a_{0,1}}  &  G_0\ar[r]_{a_{0,2}}&\cdots\ar[r]_{a_{0,n}}&G_0  }
\]
is commutative, functoriality of the strict transform and the definition of $L_1$ give us 
\begin{equation}\label{17.11.2014--2}\begin{split}\rho_0^\#\left[a_{0,n}\cdots a_{0,1}(L_0)\right]&=a_{1,n}\cdots a_{1,1}(\rho_0^\#(L_0))\\ & =a_{1,n}\cdots a_{1,2}(L_1).
\end{split}\end{equation}
(Recall that by construction $a_{1,1}=a_1$.) Since $a_{0,j}$ is conjugation by an element of $G_0(R)$ belonging to the image of $G_1(R)\to G_0(R)$, we can affirm that $a_{0,j}$ leaves $L_0\ot k$ invariant. Consequently, $L_0$ and $a_{0,n}\cdots a_{0,1}(L_0)$ have the same closed fibre, from which we conclude that $\rho_0$ is also the blowup of \[\left(a_{0,n}\cdots a_{0,1}(L_0)\right)\ot k.\]
Since \[G_{n+1}\stackrel{\rho_n}{\longrightarrow}\cdots\stackrel{\rho_1}{\longrightarrow}  G_1\] is, by induction hypotehsis, the spontaneous sequence of $a_{1,n}\cdots a_{1,2}(L_1)\subset G_1$ truncated at level $n$ and, as (see eqs. \eqref{17.11.2014--2}) 
\[a_{1,n}\cdots a_{1,2}(L_1)=\rho_0^\#\left[a_{0,n}\cdots a_{0,1}(L_0)\right],
\]
we obtain that 
\[
G_{n+1}\stackrel{\rho_n}{\longrightarrow}\cdots\stackrel{\rho_1}{\longrightarrow}  G_1\stackrel{\rho_0}{\longrightarrow}G_0
\]
is the spontaneous sequence associated to $a_{0,n}\cdots a_{0,1}(L_0)$ truncated at level $n+1$.  
\end{proof}

We maintain the above notations and write 
\[
C_n:=a_{0,n}\cdots a_{0,1}(L_0).
\]
\begin{prop}\label{18.11.2014--8}For each $n\ge m$, we have $C_n\ot R_m=C_m\ot R_m$ as closed subschemes of $G\ot R_m$. The limit $\g C:=\varinjlim_nC_n\ot R_n$ defines a closed formal subgroup of $\widehat G_0=\varinjlim_nG_0\ot R_n$ which is furthermore flat over $R$.  
\end{prop}
The proof relies on the following result.

\begin{lem}\label{18.11.2014--5}Let $\mathcal{L}\subset \mathcal {G}$ be a closed immersion in $\mathbf{FGSch}/R$. Let 
\[\cdots \stackrel{\sigma_0}{\longrightarrow}\mathcal{G}_0=\mathcal{G}\] 
stand for the spontaneous sequence associated to $\mathcal L$. 
Let $g_{n+1}\in \mathcal G_{n+1}(R)$. Then its image by the obvious arrow 
\[
\mathcal G_{n+1}(R)\stackrel{}{\longrightarrow} \mathcal G(R)\stackrel{}{\longrightarrow} \mathcal G(R_{n})
\] 
lies in $\mathcal L(R_{n})$. 
\end{lem}

\begin{proof}
To help with the verification, we employ the following observation. 

\begin{lem}\label{18.11.2014--7}
Let $\mathcal L\subset \mathcal G$ be a closed immersion in $\mathbf{FGSch}/R$. If $J$ stands for the ideal of $\mathcal L$ and $\sigma:\mathcal G'\to \mathcal G$ for the blowup of $\mathcal L\ot k$, then the ideal of the strict transform $\sigma^\#(\mathcal L)$ of $\mathcal L$ contains $\pi^{-1}J$. 
\end{lem}
\begin{proof}Evident since the ideal of the strict transform is the saturation $\cup_\nu(JR[\mathcal G']:\pi^\nu)$. 
\end{proof}

We carry on the verification of the lemma. Write $\mathcal L_0=\mathcal L$ and define $\mathcal L_{n+1}$ as the strict transform of $\mathcal L_n$ in $\mathcal G_{n+1}$. 
Let $J_n$ stand for the ideal of $\mathcal L_n$ in $R[\mathcal G_n]$.  From Lemma \ref{18.11.2014--7}, we know that  for any $f_0\in J_0$, there exists an element $f_{n+1}\in R[\mathcal G_{n+1}]$ (belonging to $J_{n+1}$) such that $\pi^{n+1} f_{n+1}=f_0$. Consequently, if $g_{n+1}:R[\mathcal G_{n+1}]\to R$ is a morphism of $R$-algebras, we conclude that $g_{n+1}(f_0)=\pi^{n+1} g_{n+1}(f_0)$, so that the morphism $g_0:R[\mathcal G_0]\to R$ associated to it satisfies $g_0(J_0)\subset(\pi^{n+1})$. Hence, $g_0\ot R_{n}:R_{n}[\mathcal G_0]\to R_{n}$  annihilates $J_0$, i.e. the corresponding point belongs to $\mathcal L_0(R_{n})$. 
\end{proof}

\begin{proof}[Proof of Proposition \ref{18.11.2014--8}]We wish first to verify 
\begin{equation}\label{24.11.2014--1}
C_{n+1}\ot_RR_n=C_n\ot_RR_n.
\end{equation}
Proposition \ref{24.10.2014--1} ensures  us that for each fixed $n$, 
\[
G_{n+1}\stackrel{\rho_n}{\longrightarrow}\cdots \stackrel{\rho_0}{\longrightarrow}G_0=G
\]
is the spontaneous sequence of  $C_n\subset G_0$ truncated at level $n+1$. We know that the inner automorphism $a_{0,n+1}:G_0\to G_0$, which sends $C_{n}$ isomorphically to $C_{n+1}$, is conjugation by an element of the form
\[
\rho_n\circ\cdots\circ\rho_0(g_{n+1}),\qquad g_{n+1}\in G_{n+1}(R).
\]
Hence, due to Lemma \ref{18.11.2014--5}, it is true that 
\[
a_{0,n+1}(C_n\ot_R R_n)= C_n\ot_R R_n.
\]
Therefore, 
\[
\begin{split}C_{n+1}\ot_RR_{n}&=a_{0,n+1}(C_n)\ot_RR_n\\&=a_{0,n+1}(C_n\ot R_n)\\&=C_n\ot R_n.
\end{split}\]

Another way of expressing equality \eqref{24.11.2014--1} is by saying that, if  $\mathfrak k_n$ stands for the kernel of $R_n[G_0]\to R_n[C_n]$, then $R_{n+1}[G_0]\to R_n[G_0]$ takes $\g k_{n+1}$ onto $\g k_n$. Consequently, $\varprojlim_nR_n[G_0]\to\varprojlim_nR_n[C_n]$ is surjective and $\g C$ is a closed formal subscheme of $\widehat G_0$. We omit the verification that $\g C$ is a subgroup of $\widehat G_0$ and refer the reader to \cite[22.3, p.174]{matsumura} for the statement about flatness.   
\end{proof}

Everything is now in place for the proof of 

\begin{thm}\label{24.11.2014--3} Let 
\[
\cdots\stackrel{}{\longrightarrow}G_{n+1}\stackrel{\rho_n}{\longrightarrow}\cdots\stackrel{\rho_0}{\longrightarrow}G_0=G
\]
be the almost spontaneous sequence (see Definition \ref{18.11.2014--3})
associated to $L_0\subset G_0$ as defined on page \pageref{24.11.2014--4}.  
If $\g C=\varinjlim C_n\ot R_n$ stands for the formal closed subgroup of $\widehat G_0$ obtained by means of Proposition \ref{24.10.2014--1} and Proposition \ref{18.11.2014--8}, then 
the standard sequence above is the the standard sequence of $\cn_{\g C}^\infty(G)\to G$. 
 
\end{thm}

In order to provide a clear proof of Theorem \ref{24.11.2014--3}, we need the following lemma. 

\begin{lem}\label{17.11.2014--5}Let $G'\to G$  and  $H\to G$ be respectively an isomorphism and a closed immersion of $\mathbf{FGSch}/R$. Denote by $H'\to G'$ the closed immersion corresponding to $H\to G$ and let $n$ be an integer. 

Assume that for some $\mu\in\mathbf N$, the standard sequence of $\cn^\mu_H(G)\to G$ coincides with the spontaneous sequence associated to $H$ when both are truncated at level $n+1$. Then the same property holds for  $\cn_{H'}^\mu(G')\to G'$. \qed
\end{lem}

\begin{proof}[Proof of Theorem \ref{24.11.2014--3}]We fix some $n\ge0$ and show that the spontaneous sequence of $\g C\subset\widehat G$ truncated at level $n+1$ is
\[
G_{n+1}\stackrel{\rho_n}{\longrightarrow}\cdots\stackrel{\rho_0}{\longrightarrow}G_0=G.
\]
The goal is to apply Proposition \ref{24.10.2014--1}.
To avoid repetitions  we let ``$\sigma_{n+1}$'' abbreviate  ``standard sequence truncated at level $n+1$.'' We also omit reference to $G$ when possible.  

Let $\mu$ be a positive integer satisfying the following properties. 
\begin{enumerate}\item[\textbf{P1.}] The $\sigma_{n+1}$ of $\cn^\infty_{\g C}$  coincides with the $\sigma_{n+1}$ of $\cn^\mu_{\g C}$. 
\item[\textbf{P2.}] The spontaneous sequence of $L_0\subset G_0$ truncated at level $n+1$ coincides with the $\sigma_{n+1}$ of $\cn_{L_0}^\mu$.
\item[\textbf{P3.}] $\mu\ge n$.  
\end{enumerate}
That $\mu$ exists is a consequence of Corollary \ref{28.10.2014--2} and  Theorem \ref{24.11.2014--5}. 
From \textbf{P1}, the $\sigma_{n+1}$ of $\cn^\infty_{\g C}$ is the $\sigma_{n+1}$ of $\cn^\mu_{\g C}=\cn^\mu_{C_\mu}$.  From \textbf{P2} and Lemma \ref{17.11.2014--5}, the $\sigma_{n+1}$ of $\cn_{C_\mu}^\mu$ is the truncation of the spontaneous sequence of $C_\mu\subset G_0$ at level $n+1$. 
  Due to Proposition \ref{24.10.2014--1} and \textbf{P3}, the spontaneous sequence associated to $C_\mu\subset G_0$, truncated at level $n+1$, is just what we started with:
\[
G_{n+1}\stackrel{\rho_n}{\longrightarrow}\cdots\stackrel{\rho_0}{\longrightarrow}G_0=G.
\] 
\end{proof}

We then obtain our main converse of Theorem \ref{26.09.2014--7} as a consequence of Proposition \ref{17.11.2014--1}, Definition \ref{18.11.2014--3} and Theorem \ref{24.11.2014--3}.  

\begin{cor}\label{09.04.2015--3} Let 
\[
\cdots\aro G_{n+1}\stackrel{\rho_n}{\aro}\cdots\stackrel{\rho_0}{\aro}G_0=G
\]
be a standard sequence (Definition \ref{19.11.2014--1}). Denote the center of the Neron blowup $\rho_n$ by $B_n\subset G_n\ot k$. Assume that the four hypothesis in Proposition \ref{17.11.2014--1} hold: 
\begin{enumerate} 
\item[i)] The group $G_0$ is smooth over $R$, and   $B_0$ is smooth over $k$. 
\item[ii)] For every $n\ge0$, the restriction  $\rho_{n}\ot k:B_{n+1}\to B_n$ is an isomorphism. 
\item[iii)] There exists an $R$-flat closed subgroup $L_0\subset G_0$ with $L_0\ot k=B_0$.  
\item[iv)] For each representation $V$ of $B_0$, the first cohomology group $H^1(B_0;V)$ vanishes (linear reductivity).  
\end{enumerate}

Then the above standard sequence is the standard sequence  of $\mathcal N_{\g C}^\infty(G)\to G$, where $\g C$ is a formal, $R$-flat,  closed subgroup scheme of $\widehat G$. \qed
\end{cor}

\section{Group schemes over $R$ in   differential Galois theory}\label{09.03.2015--9}

We now wish to apply the theory so far developed to study differential Galois theory.  
Let $f:X\to S$ be a smooth morphism  between locally noetherian schemes. We let $\cd(X/S)$ be the ring of differential operators which in [EGA $\mathrm{IV}_4$, 16.8], respectively \cite[\S2]{BO}, is denoted by $\cd iff_{X/S}$, respectively  $\cd iff(\co_X,\co_X)$.
We let 
\[
\dmod{X/S}
\]
stand for the category of (sheaves of) $\cd(X/S)$-modules whose underlying $\co_X$-module is coherent. In fact, in the present work, a $\cd(X/S)$-module will always mean an object of $\dmod{X/S}$.  
The obvious action of $\cd(X/S)$ on $\mathcal O_X$ gives rise to an object of $\dmod{X/S}$ which is denoted by $\mathbf1$; arrows in $\mathrm{Hom}_{\cd(X/S)}(\ce,\ce')$ will  frequently be called horizontal morphisms between $\ce$ and $\ce'$, and, in the particular case where $\ce=\mathbf1$, these will be named horizontal sections of $\ce'$. 
It is obvious that $\dmod{X/S}$ is an abelian category and that the evident functor from $\dmod{X/S}$ to coherent modules is exact and faithful. Furthermore, the tensor product  of coherent modules induces a tensor product (denoted simply by $\ot_{\co_X}$ or $\ot$) in $\dmod{X/S}$. 

\begin{rmk}\label{relation_remark}In order to render referencing more effective, we inform the reader that $\dmod{X/S}$ is frequently denoted by $\mathbf{str}(X/S)$ (see \cite{DH}, \cite{dS09}, \cite[\S2]{BO}). Also, as Theorem 2.15 of \cite{BO} explains, if $S$ is a $\mathbf Q$-scheme, $\dmod{X/S}$ coincides with the category of $\co_X$-coherent modules endowed with an integrable connection in the sense of \cite[Section 1]{katz70}.  
\end{rmk}

Objects of $\dmod{X/S}$ which underlie  locally free $\mathcal O_X$-modules (vector bundles) are the objects of a full subcategory $\dmod{X/S}^o$. Note that the natural ``dualization'' in the category of vector bundles allows us to define a dualization in $\dmod{X/S}^o$; this will be denoted by a superscript $(-)^\vee$. 

The proposition below is a basic result in the theory. (For proofs, the reader is directed to \cite[Proposition 8.9]{katz70} of \cite[2.16]{BO} for the first statement and to \cite[p.40]{katz90}, \cite[p.82]{dS09} or \cite[5.1.1]{DH} for the second.)

\begin{prop}\label{05.03.2014--1}  If $S$ is the spectrum of a field, any $M\in\dmod{X/S}$ is locally free as an $\mathcal O_{X}$-module. 
If $S$ is the spectrum of $R$, any $\mathcal M\in\dmod{X/S}$ which is free of $\pi$-torsion belongs to $\dmod{X/S}^o$ (and conversely). 
\qed 
\end{prop}

\emph{We now fix some assumptions and notations which will be in force for the rest of this section}. We write 
\[S=\mathrm{Spec}\,R,\] and let  
\[f:X\aro S\]
be a smooth and geometrically connected (hence integral)  morphism of finite type admitting a section $\xi:S\to X$. (The discrete valuation ring $R$ is not assumed to be complete.)
The base-change functor induced by $\mathrm{Spec}\,k\to S$ will be denoted by a subscript ``$k$'', and whenever convenient we write ``$R$'', or ``$k$'' in place of ``$S$'' or ``$\mathrm{Spec}\,k$''.

With these conventions, another cornerstone ensues. (For proofs, the reader should consult \cite[Proposition 5.1.1]{DH} and \cite[2.16]{BO}.)

\begin{prop}\label{05.03.2015--2}The pull-back functors \[\xi^*:\dmod{X/R}\aro\mathbf{mod}(R)\]
and
\[\xi_k^*:\dmod{X_k/k}\aro\mathbf{mod}(k)\] 
are exact and faithful.  
\qed
\end{prop}

In  possession of these facts, we can put forward the main definitions of this section. 
\begin{defn}\label{19.06.2015--2} 
\begin{enumerate}
\item Let $\mathcal M$ be an object of $\dmod{X/S}^o$. Write $\mathbf T^{a,b}\mathcal M$ for the $\cd(X/S)$-module $\mathcal M^{\ot a}\ot \mathcal M^{\vee\ot b}$ and denote by $\langle \mathcal M\rangle_\ot$ the full subcategory of $\dmod{X/R}$ having as objects subquotients of direct sums $\mathbf{T}^{a_1,b_1}\mathcal M\oplus\cdots\oplus\mathbf{T}^{a_m,b_m}\mathcal M$ for varying $a_i,b_i$. 
  
\item  Let $\alpha:\mathcal M'\to \mathcal M$ be a monomorphism in $\dmod{X/R}$ with both $\mathcal M$ and $\mathcal M'$ locally free as $\mathcal O_X$-modules. If $\mathrm{Coker}(\alpha)$ is also locally free, we say, following \cite[Definitions 10 and 23]{dS09}, that $\alpha$ is a special monomorphism.  
Call an object $\mathcal M''\in \dmod{X/R}^o$ a special sub-quotient of $\mathcal M$ if there exists a special monomorphism $\mathcal M'\to \mathcal M$ and an epimorphism $\mathcal M'\to \mathcal M''$. The category of all special sub-quotients of various $\mathbf{T}^{a_1,b_1}\mathcal M\oplus\cdots\oplus\mathbf{T}^{a_m,b_m}\mathcal M$ is denoted by $\langle \mathcal M\rangle_\ot^s$.   
\end{enumerate}
\end{defn}

The structure result which will enable us to see the theory of $\cd$-modules through group theoretical lenses is the following. (The proof, which is in Saavedra's seminal work \cite{saavedra}, is explained concisely in \cite{DH}.)

\begin{thm}\label{06.03.2015--1}Let $\mathcal M\in\dmod{X/R}^o$. The $R$-point $\xi$ induces an equivalence of abelian tensor  categories  
\[
\xi^*:\langle\mathcal M\rangle_\ot\aro\mathrm{Rep}_R(\mathrm{Gal}'(\mathcal M)),
\]
where $\mathrm{Gal}'(\mathcal M)$ is a flat group scheme over $R$. 
\qed
\end{thm}

\begin{defn}\label{21.08.2015--1}The group scheme $\mathrm{Gal}'(\mathcal M)$ is called the full differential Galois group of $\mathcal M$.
\end{defn}

\begin{defn}\label{21.08.2015--2}Let $\mathcal{M}\in\dmod{X/R}^o$ and write 
\[
\rho:   \mathrm{Gal}'(\mathcal{M})\aro \mathbf{GL}(\xi^*\mathcal{M})    
\]
for the associated representation of the full differential Galois group. The restricted differential Galois group of $\mathcal{M}$, denoted $\mathrm{Gal}(\mathcal M)$, is the group scheme $\Psi_\rho$ of Definition \ref{21.08.2014--1}. 
\end{defn}

Consider the diptych (Definition \ref{21.08.2014--1}) of $\rho$:
\[
\xymatrix{\Psi_\rho' \ar[r]& \mathrm{Gal}(\mathcal M)\ar@{^{(}->}[d]  \\\mathrm{Gal}'(\mathcal M)\ar[u]\ar[r]_\rho & \mathbf{GL}(\xi^*\mathcal{M}). }
\]
From Proposition \ref{09.03.2015--2}(2),  Definition \ref{19.06.2015--2}(1) and \cite[Theorem 4.1.2]{DH} we know that the leftmost arrow above is an \emph{isomorphism}. Of course, due to this, one may think that we have chosen an improper way to present things; the reader   who wishes to complete the  diagram by including the fundamental group scheme $\Pi(X/S,\xi)$ (see \cite[Definition 5.1.4]{DH}) in the lower left corner is invited to do so at the expense of having to understand  that $\mathrm{Rep}_R(\Pi)$ is not what one might naively think it is. Since we focus on differential Galois groups and not fundamental group schemes, we leave $\Pi(X/S,\xi)$ as inspirational.

Then, as in Section \ref{06.03.2015--2} we have:

\begin{prop}\label{19.06.2015--1}The following claims are true. 
\begin{enumerate}\item The arrow $\mathrm{Gal}'(\mathcal M)\to \mathrm{Gal}(\mathcal M)$ induces an isomorphism on generic fibres and is  a possibly infinite iteration of Neron blowups. 
\item The functor $\xi^*$ induces an equivalence  \[\langle\mathcal{M}\rangle_\ot^s\stackrel{\sim}{\aro}\mathrm{Rep}_R(\mathrm{Gal}(\mathcal{M}))^o.\] \end{enumerate}\qed
\end{prop}

We are then ready to present the ``empiric'' version of Example \ref{15.10.2014--2}. 

\begin{ex}\label{15.10.2014--1}Assume that  $k$ is of characteristic zero and that $R$ contains a copy of it.  Let $A=R[x,1/x]$ and define $X$ to be $\mathrm{Spec}\,A$. In this case, the structure of $\cd(X/S)$ is quite simple,  and to give the free $\co_X$-module $\mathcal L=\mathcal O_X \mathbf e$ the structure of a $\cd(X/S)$-module we only need to specify the action of $d/dx=:\partial$ (see 2.6 and 2.15 of  \cite{BO} for details). We put  
\[
\partial\mathbf e=-\frac{\pi}{x}\mathbf e 
\]
and set out to compute $\mathrm{Gal}(\mathcal L)$ and $\mathrm{Gal}'(\mathcal L)$ (at some unspecified $R$-point).

In $\langle\mathcal{L}\rangle_\otimes$ we have the $\cd(X/R)$-module $\mathcal V$ given by 
\[\begin{split}
\partial \mathbf v_1&=\frac{1}{x}(\mathbf v_2-\pi\mathbf v_1)\\
\partial\mathbf v_2&=0;
\end{split}
\]
it is the sub-object of $\mathbf1\oplus\mathcal{L}$ generated by $\mathbf v_1=(1,\mathbf e)$ and $\mathbf v_2=(\pi,0)$. (Needless to say, the construction of $\mathcal V$ parallels that in Proposition \ref{05.05.2014--2}.) Since $\pi(\mathbf1\oplus\mathcal{L})\subset\mathcal V$, we conclude that $\langle  \mathcal V \rangle_\otimes=\langle  \mathcal{L} \rangle_\otimes$. We denote by $\mathrm{Gal}'$ the full differential Galois group of $\mathcal L$, which is the same as that of $\mathcal V$, at some $R$-point, so that we have, associated to $\mathcal{L}$, the representation $\rho:\mathrm{Gal}'\to\mathbf{G}_{m,R}$. Let  $\GG_{m,R}'$ stand for the Neron blowup of $\mathbf G_{m,R}$ at the origin of the closed fibre. Since $\rho_k$ is the trivial representation (it corresponds to $\mathcal{L}_k$), we derive a factorization 
\[
\xymatrix{\mathrm{Gal}'\ar[r]^{\rho'}\ar[dr]_\rho & \mathbf G_{m,R}'  \ar[d]^\gamma \\ & \mathbf G_{m,R}   }
\] of $\rho$. Under this factorization, the  faithful representation of $\mathbf G_{m,R}'$ constructed in Proposition \ref{05.05.2014--2}  is taken to $\mathcal V$. In particular, we have a faithful representation of $\GG_{m,R}'\ot k$ which is taken to $\mathcal V_k \in\langle \mathcal L\rangle_\ot$.  
Now we note that $\mathcal{V}_k$ is the module of the logarithm, so that $\langle\mathcal V_k\rangle_\ot\stackrel{\sim}{\to}\mathrm{Rep}_k(\mathbf{G}_{a,k})$. We arrive at a commutative diagram 
\[
\xymatrix{ \mathrm{Gal}'_k\ar[r]^-{\rho'_k} \ar@{->>}[d]& \mathbf{G}_{m,R}'\otimes k\\ \mathbf{G}_{a,k}.\ar[ur]   }
\]
Since the arrow \[\mathrm{Rep}_k(\mathbf{G}_{m,R}'\otimes k)\longrightarrow\mathrm{Rep}_k(\mathbf{G}_{a,k})\] takes a faithful representation of $\mathbf{G}_{m,R}'\otimes k$ to a faithful representation of $\mathbf{G}_{a,k}$, we conclude that $\mathbf{G}_{a,k}\to \mathbf{G}_{m,R}'\otimes k$ is a closed embedding. As $\mathbf{G}_{m,R}'\otimes k=\mathbf{G}_{a,k}$ and $k$ is of characteristic zero,  $\mathbf{G}_{a,k}\to \mathbf{G}_{m,R}'\otimes k$ is an isomorphism. 
This guarantees that $\rho'$ is is an isomorphism \cite[Lemma 1.3, p.551]{WW} so that  the diptych of $\rho$ (see Section \ref{06.03.2015--2}) is 
\[
\xymatrix{  \mathbf{G}_{m,R}'\ar[r]^\gamma& \mathbf{G}_{m,R}\ar[d]^{\mathrm{id}}   \\ \ar[u]^{\rho'} \mathrm{Gal}'\ar[r] & \mathbf{G}_{m,R}  . }
\] 
To summarise, $\mathrm{Gal}'=\mathbf{G}_{m,R}'$ and $\mathrm{Gal}=\mathbf{G}_{m,R}$. 
When reduced modulo $\pi$, this gives 
\[
\xymatrix{  \mathbf{G}_{a,k}\ar[r]& \mathbf{G}_{m,k}\ar[d]^{\mathrm{id}}   \\ \ar[u]^{\sim} \mathrm{Gal}'_k\ar[r]_{\rho_k} & \mathbf{G}_{m,k}, }
\]
in which the image of $\rho_k$ is the trivial group \cite[8.3, Corollary, p.65]{waterhouse}. 
\end{ex}

\begin{rmk}Note that this example shows that \cite[Theorem 3.3.1.1, p. 729]{Andre} cannot be true. In fact, what in \cite{Andre} is defined as $\mathrm{Gal}(\mathcal L)$, respectively $\mathrm{Gal}(\mathcal L_k)$,  is what we here call $\mathrm{Gal}'$, respectively $\mathrm{Im}(\rho_k)$,  and we have showed that   $\mathrm{Gal}'_k\not\simeq\mathrm{Im}(\rho_k)$.
  It seems that the inconsistency in the proof lies in the definition of the arrow $u$ in \cite[3.2.3.4, p.727]{Andre}. 
\end{rmk}

\begin{ex}\label{10.11.2016--1}Assume $k$ to be of characteristic zero and that $R$ contains a copy of it. Let $X=\mathrm{Spec}\,R[x]$ and define $\mathcal L$ to be the free $\co_X$-module on the element $\mathbf e$. 
We then introduce on $\mathcal L$ the structure of $\cd(X/R)$-module by fixing  
\[\frac{\partial}{\partial x}\mathbf e=-\pi\mathbf e.
\]
(The reader should consult 2.15 of \cite{BO} to understand this construction.)
Since this is just the differential equation of the ``function'' $\exp(-\pi x)$, we are led to consider the following identity: 
\[
\frac\partial{\partial x} \left( \sum_{\nu=0}^n\frac{(\pi x)^\nu}{\nu!}\mathbf e \right)=-\frac{x^n}{n!}\pi^{n+1}\mathbf e.
\]
Consequently, we see that 
$\mathcal L\ot R_n\simeq \mathbf1\ot R_n$. 
Using Corollary \ref{15.06.2015--1}, we conclude that $\mathrm{Gal}'(\mathcal L)$ is the automatic blowup of $\mathrm{Gal}(\mathcal L)=\GG_{m,R}$ at $\{e\}$. 
\end{ex}

\section{Coincidence of Galois groups in the case of inflated $\cd$-modules}\label{20.08.2015--2}
We maintain the notations of Section \ref{09.03.2015--9} introduced after Proposition 
\ref{05.03.2014--1}. In particular,    $f:X\to S=\mathrm{Spec}\,R$ is a smooth and geometrically connected (hence integral) morphism admitting a section $\xi:S\to X$. \emph{We also assume that $k$ is algebraically closed and that $R$ contains a copy of it.}

We now wish to study the differential Galois groups of certain objects in $\dmod{X/S}$ of a very special kind: \emph{inflations}.   
As $X$ is certainly not of finite type over $k$---an assumption which is usually present in the literature, see for example  \cite[Section 1]{katz70} or \cite[\S2]{BO}---we shall need a few adjustments to our theory.

\subsection{Approximating $R$ by smooth $k$-algebras of finite type}\label{28.10.2016--1}
When $R$ is the Zariski or etale local ring of a closed point on a smooth curve over $k$, it is possible to find  a \emph{filtered} family of $k$-subalgebras $\{R_\lambda\}$ of $R$ enjoying the ensuing properties. 
\begin{enumerate}\item[i)] The $k$-scheme $S_\lambda:=\mathrm{Spec}\,R_\lambda$ is smooth and  
\item[ii)] the natural arrow $S\to\varprojlim_\lambda S_\lambda$ is an isomorphism. 
\end{enumerate}
On the other hand,  such a family of $k$-subalgebras invariably exists as can be seen from the arguments given by  Artin  on pages 40 and 41 of \cite{artin_approximation}. 
(Unfortunately the \emph{statements} of the results in  \cite[Section 4]{artin_approximation} do not immediately apply in our setting.  
Since it seems pointless to present a quick survey of the concerned theme, we simply leave the reader with the previous reference and the task of examining the mentioned pages.)  

As explained in  \cite{EGA}, see $\mathrm{IV}_3$ 8.8.2(ii) p.28 and $\mathrm{IV}_4$, 17.7.8(ii) p.75, there exists some index $\alpha$ and a smooth morphism of finite type 
\[
f_\alpha:X_\alpha\aro S_\alpha
\]
giving back $f:X\to S$. In addition, \cite{EGA} $\mathrm{IV}_3$, 12.2.4, p.183 lets us assume that $f_\alpha$ is geometrically integral. In fact, many other properties of $f$ are reflected starting from some model on, see \cite[$\mathrm{IV}_3$, 8.10.5]{EGA} for example.

In order to ease communication in what comes, for $\lambda\ge\alpha$, we call $X_\lambda$, $f_\lambda$, etc,  a \emph{model of $X$}, $f$, etc, and we write models as well as ``objects associated to these'' using bold letters instead of putting subscripts. 

Finally, two more conventions will prove useful:  
A certain $X_\mu$ is said to be \emph{thiner} than $X_\lambda$ if $\mu\ge\lambda$ 
if $\mu\ge\lambda$, and for each model $\bm f:\bm X\to \bm S$ of $f:X\to S$, we write $\bm X_k$ for the fibre of $\bm f$ above the point 
\[
\xymatrix{\mathrm{Spec}\,k\ar[rr]^-{\text{closed point}}&& S\ar[rr]^{\text{canonic}}&& \bm S. } 
\]

\subsection{Inflating $\cd$-modules}
Let $\bm f:\bm X\to \bm S$ be a model of $f:X\to S$ and let  $\bm\cm\in\dmod{\bm X/k}$. Since on $\bm X$ we have a morphism of left $\co_{\bm X}$-algebras 
\[
\cd(\bm X/\bm S)\aro\cd(\bm X/k),
\]
we can give $\bm{\cm}$ the structure of a $\cd(\bm X/\bm S)$-module. This association, which is clearly functorial, is called the inflation functor and denoted by $\mathrm{Inf}(\bm\cm)$. 
Analogously, we define the inflation of $\bm\cm$ to $X$, call it $\mathrm{Inf}_X(\bm\cm)$, as being the $\cd(X/S)$-module on $X$ induced by $\mathrm{Inf}(\bm\cm)$ via the canonical morphism $X\to \bm X$. The reader who is unfamiliar with the process of pulling-back $\cd$-modules will find it made clear in \cite[2.1]{BeDMA2}. 

The following result will prove useful.

\begin{lem}\label{26.10.2016--1} Let $\bm\cm\in\dmod{\bm X/k}$ and $\bm\cn\in\langle \bm\cm \rangle_\ot$.
Then,  $\mathrm{Inf}_X(\boldsymbol{\mathcal N})\in\langle\mathrm{Inf}_X(\bm\cm )\rangle_\ot^s$. 
\end{lem}

\begin{proof}
This follows immediately from the fact that all $\bm{\mathcal C}$ in $\dmod{\bm X/k}$ underlie locally free $\mathcal O_{\bm X}$-modules (see Proposition \ref{05.03.2014--1}). 
\end{proof}

\subsection{The differential Galois group of an inflated $\cd$-module}
\begin{thm}\label{absolute-conn}We assume in addition that $f:X\to S$ is projective. Let $\boldsymbol{\mathcal M}$ be an object of $\dmod{\bm X/k}$ for some model $\bm X$ of $X$, and let  $\mathcal M=\mathrm{Inf}_X(\bm\cm)$. Then the morphism of group schemes over $R$
\[
\mathrm{Gal}'(\mathcal M)\aro \mathrm{Gal}(\mathcal M)
\] 
is an isomorphism. 
\end{thm}

\begin{ex}\label{09.11.2016--1}Here is an example showing that the Theorem \ref{absolute-conn} \emph{does not} hold if $X$ is simply affine. 
Let  $\bm X=\mathrm{Spec}\,\mathbf C[\pi,x]$ be the affine line over $\mathbf C[\pi]$, and $\boldsymbol{\mathcal L}=\mathcal O_{\bm X}\mathbf e$ the free $\mathcal O_{\bm X}$-module on $\mathbf e$.
We induce on $\bm{\mathcal L}$ the structure of a $\cd(\bm X/\mathbf C)$-module by 
\[
\begin{split}
\frac\partial{\partial x}\mathbf e&=-\pi \mathbf e\\
\frac\partial{\partial \pi}\mathbf e&=x\mathbf e.
\end{split}
\] (The reader should consult 2.15 of \cite{BO} to understand this construction.)
Its inflation $\mathcal L$ is then defined by 
\[
\frac\partial {\partial x}\mathbf e=-\pi\mathbf e.
\]
As we saw in Example \ref{10.11.2016--1}, $\mathrm{Gal}'(\mathcal L)$ is the automatic blowup of $\mathrm{Gal}(\mathcal L)=\GG_{m,R}$ at $\{e\}$. 
\end{ex}

Our proof of Theorem \ref{absolute-conn} depends on two known results (found in \cite{EH1}, \cite{dS12}, \cite{Zha14}). Since some work has to be done in order to adapt the substance of these theorems to our setting, we prefer to explain them in Section \ref{03.03.2015--2} and give a proof of Theorem \ref{absolute-conn} now.

\begin{proof}[Proof of Theorem \ref{absolute-conn}] To lighten notation, we write  $G$ and $G'$ in place of $\mathrm{Gal}(\mathcal M)$ and $\mathrm{Gal}'(\mathcal M)$, respectively. We also make use of the notations and conventions introduced on Section \ref{28.10.2016--1}.

It is enough to prove that the functor
$$\mathrm{Rep}_k(G_k)\longrightarrow \langle \mathcal M|_{X_k}\rangle_\otimes$$
is an equivalence (cf. Corollary \ref{09.05.2014--2}(i) and diagram \eqref{09.03.2015--8}).
The proof is a consequence of the  following claims.

\emph{Claim 1.} Let $\boldsymbol \cn$ and $\boldsymbol\cn '$ be objects in $\dmod{\boldsymbol X/k}$ and let 
\[
\xymatrix{
\mathrm{Inf}_X(\bm \cn)|_{X_k}\ar[rr]^-{\theta} && \mathrm{Inf}_X(\boldsymbol \cn')|_{X_k} \\ 
\bm \cn|_{\bm X_{k}}\ar@{=}[u] &&  \bm\cn'|_{\bm X_{k}}\ar@{=}[u] 
}
\]
be an arrow of $\cd(X_k/k)$-modules. 
Then there exists a morphism of $\cd(X/S)$-modules  
\[\widetilde \theta: \mathrm{Inf}_X(\boldsymbol\cn)\aro\mathrm{Inf}_X(\boldsymbol\cn')\] lifting $\theta$. 

\emph{Verification.} We have an arrow of $\cd(\boldsymbol X_k/k)$-modules 
\[
\theta: \boldsymbol\cn|_{\boldsymbol X_k} \aro \boldsymbol\cn'|_{\boldsymbol X_k}
\]
which gives us an arrow of $\cd(\boldsymbol X_k/k)$-modules 
\[
\sigma:\co_{\boldsymbol X_k} \aro \boldsymbol\cn|_{\boldsymbol X_k}^\vee \ot \boldsymbol\cn'|_{\boldsymbol X_k}=:\boldsymbol \ce|_{\boldsymbol X_k}.
\]
We will show that $\sigma$ is the restriction of an arrow $\co_{\bm X}\to \bm\ce$ of $\dmod{\bm X/\bm S}$.
Now, let \[\iota:T\aro \boldsymbol\ce|_{\boldsymbol X_k}\] be the maximal trivial subobject; the arrow  $\sigma$  can therefore be written as a composition in $\dmod{\bm X_k/k}$
\begin{equation}\label{19.10.2016--4}
\co_{\boldsymbol X_k}\stackrel{\tau}\aro T\stackrel{\iota}{\aro} \boldsymbol\ce|_{\boldsymbol X_k}
\end{equation}
According to Theorem \ref{lem.santos}-(1), it is possible to find $\boldsymbol\ct\in\dmod{\boldsymbol S/k}$ and a morphism of $\cd(\boldsymbol X/k)$-modules 
\[
\widetilde\iota:\boldsymbol f^*(\boldsymbol \ct) \aro  \boldsymbol\ce
\] 
such that $\iota$ is the restriction to $\boldsymbol X_k$ of  $\wt\iota$.
As $\boldsymbol f$ is proper, flat and geometrically integral, we have $\bm f_*\co_{\bm X}=\co_{\bm S}$ \cite[X, Proposition 1.2, p.202]{SGA1}; it then follows that the functor $\bm f^*$ from vector bundles on $\bm S$ to vector bundles on $\bm X$ is full and faithful. As $\bm S$ is affine, we conclude that the morphism of $\co_{\boldsymbol X_k}$-modules 
\[
 \bm f^*(\co_{\bm S})|_{\bm X_k}=\co_{\bm X_k}\stackrel{\tau}{\aro} T= \bm f^*(\bm\ct)|_{\bm X_k}
\] appearing in \eqref{19.10.2016--4} is the restriction of a morphism 
\[
\widetilde\tau=\bm f^*(\de) :\bm f^*(\co_{\mathbf S})\aro\bm f^*(\bm \ct).
\]
Of course, $\de$ need not be a morphism of $\cd(\bm S/k)$-modules, but $\bm f^*(\de)$ is surely a morphism of $\cd(\bm X/\bm S)$-modules, that is, an arrow between inflations 
\[
\mathrm{Inf}(\co_{\bm X})\aro \mathrm{Inf}(\bm f^*\bm \ct).
\]
In conclusion, we have proved that $\sigma$ is the restriction of $\widetilde\iota\circ\widetilde \tau$.

\noindent\emph{Claim 2.} For each $V\in\langle\cm|_{X_k}\rangle_\ot$, there exist $\ce$ and $\ce'$ in  $\langle\cm\rangle_\ot^s$ and an arrow of $\dmod{X/S}$
\[
\widetilde\theta:\ce\aro\ce'
\]
such that 
\[
V\simeq \mathrm{Coker}\left(\widetilde\theta|_{X_k}:\ce|_{X_k}\aro\ce'|_{X_k}\right)
\]

\noindent\emph{Verification.} By definition,  $\cm|_{X_k}=\bm\cm|_{\bm X_k}$. Hence, according to Theorem \ref{lem.santos}-(2) (applied to the dual of $V$), we can find 
$\bm \cn$ and $\bm \cn'$ in $\langle  \bm \cm  \rangle_\ot$
fitting into an exact sequence of $\cd(\bm X/k)$-modules: 
\[
\bm \cn|_{\bm X_k}\stackrel{\theta}{\aro}\bm \cn'|_{\bm X_k}\aro V\aro0.
\]
Using Claim 1, $\theta$ is the restriction to $X_k$ of an arrow in $\dmod{X/S}$
\[
\widetilde \theta: \mathrm{Inf}_X(\bm \cn)\aro  \mathrm{Inf}_X(\bm \cn'). 
\]
We then take $\mathcal E=\mathrm{Inf}_X(\bm\cn)$ and $\mathcal E'=\mathrm{Inf}_X(\bm\cn')$, and the proof is finished since these do belong to $\langle\mathcal M\rangle_\ot^s$ (see Lemma \ref{26.10.2016--1}).

\noindent\emph{Claim 3.} Denote by $\eta$ the  composition of functors  
\[\mathrm{Rep}_R(G)\aro \mathrm{Rep}_R(G')\stackrel{\sim}{\aro}  \langle \cm \rangle_\ot.
\]
For each $V\in\mathrm{Rep}_k(G_k)$,
there exists $N\in\mathrm{Rep}_R(G)^o$ 
such that 
\begin{enumerate}\item[(a)] $V$ is a quotient of $N_k$ and 
\item[(b)] after eventually passing to a thiner model of $X$, there exists some 
$\boldsymbol{\mathcal N}\in\langle\boldsymbol{\mathcal M}\rangle_\ot$ such that $\eta(N)=\mathrm{Inf}_X(\boldsymbol{\mathcal N})$.   
\end{enumerate}

\emph{Verification.} According to \cite[Proposition 3, p.41]{serre} we can ``almost lift'' $V$. Precisely, there exists $E\in\mathrm{Rep}_R(G)^o$ and a   surjection $E_k\to V$. 
By means of the equivalence 
\[
\eta :\mathrm{Rep}_R(G)^o\stackrel{\sim}{\aro} \langle \cm \rangle_\otimes^s
\] 
of Proposition \ref{19.06.2015--1}, we can find a diagram in $\dmod{X/S}$:
\[\xymatrix{\mathcal F\ar[r]\ar[d]& \mathcal T\\ \eta(E),}\]
where $\mathcal T$ is some tensor power (see Definition \ref{19.06.2015--2}-(1)) of $\cm=\mathrm{Inf}_X(\bm\cm)$, the vertical arrow is an epimorphism (in $\dmod{X/S}$), and the horizontal arrow is special (see Definition \ref{19.06.2015--2}). In particular, $\ct=\mathrm{Inf}_X(\bm \ct)$ for some tensor power $\bm \ct$ of $\bm\cm$.  
According to Theorem \ref{lem.eh}, eventually passing to a thiner model of $X$, there exists  \[\bm\cn\in\langle \bm{\mathcal M} \,:\, \dmod{\bm X/k}\rangle_\otimes\] and an epimorphism  
\[
\mathrm{Inf}_X(\bm{\mathcal N})\aro \mathcal F. 
\] Using Lemma \ref{26.10.2016--1}, we conclude that $\mathrm{Inf}_X(\bm\cn)$ in fact belongs to $\langle\cm\rangle^s_\ot$; the above equivalence then produces the desired $N$, viz. any object of $\in\mathrm{Rep}_R(G)^o$ which is taken by $\eta$ to $\mathrm{Inf}_X(\bm\cn)$. Indeed, (b) is verified by construction, and (a) follows from the fact that if $\eta(\theta):\eta(N)\to\eta(E)$ is an epimorphism of $\cd(X/S)$-modules, then $\theta$ is an epimorphism in $\mathrm{Rep}_R(G)$ (between objects of $\mathrm{Rep}_R(G)^o$).

\emph{Claim 4.} The functor \[\eta_k: \mathrm{Rep}_k(G_k)=\mathrm{Rep}_R(G)_{(k)}\aro \langle \cm \rangle_{\ot,(k)}
\]
is full.

\emph{Verification.} 
Let $\varphi: \eta_k (V)\to \eta_k(V')$ be a 
morphism in $\dmod{X_k/k}$. It then fits into a commutative diagram 
\[
\xymatrix{\eta(N)\ot k\ar[r]^\theta\ar@{->>}[d] & \eta(N')\ot k  \\ \eta_k(V)\ar[r]_\varphi&\eta_k(V'),\ar@{^{(}->}[u]}
\]
where $\eta(N)=\mathrm{Inf}_X(\bm\cn)$ and $\eta(N')=\mathrm{Inf}_X(\bm\cn')$ are constructed from Claim 3. 
Claim 1 gives us  a lift 
\[
\widetilde\theta:\eta(N)\aro \eta( N')
\]
of $\theta$.  Lemma \ref{26.10.2016--1} allied to the fact that $\bm\cn$ and $\bm \cn'$ belong to $\langle\bm \cm\rangle_\otimes$ shows that   both $\mathrm{Inf}_X(\bm\cn)$ and $\mathrm{Inf}_X(\bm\cn')$ lie in $\langle\cm\rangle_\ot^s$. Since $\eta$ is an equivalence between $\mathrm{Rep}_R(G)^o$ and $\langle\cm\rangle_\ot^s$, there exists $\sigma:N\to N'$ such that $\eta(\sigma)=\widetilde\theta$.
Since the vertical  arrows in the above diagram also belong to the image of $\eta_k$, the proof of the claim is finished. 

\end{proof}

\section{Adapting two known results to our setting}\label{03.03.2015--2}

Our goal in this section is to adapt two known  results (Theorem \ref{lem.santos} and Theorem \ref{lem.eh}) to the  setting of Theorem \ref{absolute-conn}. The conventions and notations we follow are those in the beginning of Section \ref{20.08.2015--2}. In particular, $f:X\to S$ is a smooth morphism with geometrically connected (hence integral) fibres and $R$ contains a copy of the residue field $k$. We also adopt the hypothesis and terminology developed on Section \ref{28.10.2016--1}, in particular $S$ is a limit of affine smooth $k$-schemes. To handle the projectivity hypothesis on a model of $f:X\to S$ appearing in the next statement, the reader is referred to \cite[$\mathrm{IV}_3$, 8.10.5-(xiii), p.37]{EGA}.

\begin{thm}\label{lem.santos} Assume in addition that $f:X\to S$ is projective and let $\bm f:\bm X\to\bm S$ be a projective model of $f$. Let $\bm{\mathcal M}$ be an object of  $\dmod{\bm X/k}$ and write $\bm X_k$ for the fibre of $\bm f:\bm X\to\bm S$ above the $k$-point of $\bm S$ induced by the  $k$-point of $S$. 

\begin{enumerate}
\item The maximal trivial subobject of $\bm{\mathcal M} |_{\bm X_k} $ (this belongs to $\dmod{\bm X_k/k}$) is the restriction of a subobject $\bm{\mathcal T}\subset \bm {\mathcal M}$. Moreover, $\bm{\mathcal T}$ is the pull-back to $\dmod{\bm X/k}$ of an object of $\dmod{\bm S/k}$.
\item If $N$ belongs to $\langle\bm{\mathcal M}|_{X_k}\rangle_\otimes$, then there exists $\bm{\mathcal N}$ in $\langle\bm{\mathcal M}\rangle_{\ot}$ and a monic  $N\to \bm{\mathcal N}|_{\bm X_k}$.\end{enumerate}
\end{thm}
\begin{proof} If $\Pi(-)$ stands for the affine group scheme associated to the category $\dmod{\bm X/k}$, then the sequence of fundamental group schemes 
\[
\Pi(\bm X_k)\aro \Pi(\bm X)\aro\Pi(\bm S)\aro1 
\]
is exact (cf. \cite[Theorem 1.1]{Zha14} or \cite[Theorem 1]{dS12}). Using the characterisation of exactness presented in \cite[Appendix]{EHS}, we immediately arrive at the desired conclusion.  Moreover, the proof in loc.cit. shows that 
$\mathcal N$ can be chosen in $\langle  
\mathcal M\rangle_\otimes$.
\end{proof}

We now head to prove the following. 

\begin{thm}[see \cite{EH1}, Theorem 5.10]\label{lem.eh}Suppose that $f:X\to S$ is separated. Let $\bm f:\bm X\to \bm S$ be a model of $f:X\to S$ and $\bm\cm$ an object of $\dmod{\bm X/k}$. Write 
\[
\cm:=\mathrm{Inf}_X(\bm\cm)
\]
and let 
\[
\cv\aro  \cm
\]
be a special monic (see Definition \ref{19.06.2015--2}) in the category $\dmod{X/S}$. Then, eventually passing to a thiner model, there exists $\boldsymbol{\mathcal S}\in\dmod{\bm X/k}$ and an epimorphism 
\[
\mathrm{Inf}_X(\boldsymbol{\mathcal S})\aro \cv.
\]
Colloquially,  each special $\cd(X/S)$-submodule of an inflation is also a quotient of an inflation. Moreover, $\boldsymbol{\mathcal S}$ can be chosen in $\langle\boldsymbol{\mathcal M}\rangle_\otimes$.
\end{thm}

The proof hinges on the concept of socle (see Lemma \ref{24.02.2015--3}) and on a criterion which allows us to detect when a $\cd$-submodule of an inflation actually comes from an inflation (Corollary \ref{21.10.2016--4}).

\begin{lem}\label{24.02.2015--3} Let $\mathcal L$ and $\mathcal E$ be  objects of $\dmod{X/S}$. Assume that $\mathcal E$, respectively $\mathcal L$,  is locally free, respectively locally free of rank one, as an $\mathcal O_X$-module. Define the $\mathcal L$-socle of $\mathcal E$,
\[
\mathrm{Soc}(\mathcal L)\subset \mathcal E,
\]
as the sum of all subobjects of $\mathcal E$ which are isomorphic to $\mathcal L$. The following properties are true.  
\begin{enumerate}\item  The sub-object $\mathrm{Soc}(\mathcal L)$ of  $\mathcal E$ is special. 
\item In $\dmod{X/S}$ we have $\mathrm{Soc}(\mathcal L)\simeq \mathcal L^{\oplus r}$ for some $r$.
\end{enumerate} 
\end{lem}

\begin{proof}Assume that $\mathcal L$ is the \emph{trivial} $\cd(X/S)$-module $\mathbf1$. Since  $\langle\mathcal E\rangle_\ot$ is equivalent to the category of representations of an affine and flat group scheme over $R$ (see Theorem \ref{06.03.2015--1}), the result becomes a simple exercise. 
The general case is treated by employing the $\cd(X/S)$-module $\mathcal L^\vee\ot\mathcal E$.
\end{proof}

We now explain a condition which guarantees that a $\cd(X/S)$-submodule of some inflation is still an inflation. The results follow mainly from: 

\begin{thm}\label{21.10.2016--3} Let $T=\mathrm{Spec}\,\Lambda$ be a $k$-scheme possessing a system of etale coordinates $t_1,\ldots,t_m$ relative to $k$, and   $g:Y\to T$  a smooth morphism. Let $\cm\in \dmod{Y/k}$ and let $\ce\to\mathrm{Inf}(\cm)$ be a monomorphism in $\dmod{Y/T}$.   If 
\[
\mathrm{Hom}_{\mathcal D_{Y/T}}(\ce,\mathrm{Inf}(\mathcal M)/\mathcal E)=0,
\]
then $\ce$ is in fact invariant under the action of $\cd(Y/k)$ on $\cm$, that is, it is a $\cd(Y/k)$-submodule. 
\end{thm}

\begin{proof} Let  $U=\mathrm{Spec}\,A$ be an affine open subset of $Y$ on which there are  etale coordinates $x_1,\ldots,x_n$ relative to $T$. We then have etale coordinates $t_1,\ldots,t_m$ and $x_1,\ldots,x_n$ on $U$ relative to $k$. 
Let $\partial^{[q,r]}$ be the differential operators constructed from [EGA $\mathrm{IV}_4$, 16.11.2, p.54] by means of the system of etale coordinates $t_1,\ldots,t_m$, $x_1,\ldots,x_n$; here $q\in\mathbf N^m$ and $r\in\mathbf N^n$. Note that
$\partial^{[q,0]}(g^*(a))=g^*(\partial_{\bm t}^{[q]}(a))$, where $\partial^{[q]}_{\bm t}$ is the differential operator $\partial_{\bm t}^{[q]}$ constructed from the set of etale coordinates $t_1,\ldots,t_m$. This means that we have lifted $\partial_{\bm t}^{[q]}$, but the reader should note that if $\overline x_1,\ldots,\overline x_m$ is another system of etale coordinates on $U$ relative to $T$, then the differential operator obtained from $\overline x_1,\ldots,\overline x_n$, call it $\overline\partial^{[q,0]}$, is not necessarily the same as $\partial^{[q,0]}$. However, $\partial^{[q,0]}-\overline\partial^{[q,0]}$ does annihilate $\Lambda$ and induces therefore a section of $\cd(U/T)$.  

Let $\mu\in\mathbf N^m$ have zeroes all over except at one entry, where it has a 1. We define the additive map 
\[
\phi_{\mu}:\ce(U)\aro \cm(U)/\ce(U),\quad e\longmapsto \partial^{[\mu,0]}(e)\mod\ce(U). 
\]
It is easily verified that $\phi_\mu$ is $A$-linear and that 
\[
\phi_\mu(\partial^{[0,r]}e)=\partial^{[0,r]}\phi_\mu(e). 
\] 

Note that, if $\overline x_1,\ldots,\overline x_n$ is another system of etale coordinates for $U$ over $T$, employing the notation introduced above, the fact that $\partial^{[\mu,0]}-\overline\partial^{[\mu,0]}$ is a section of $\cd(U/T)$ means that $\phi_{\mu}$ is independent of the choice of the system of etale coordinates of $U$ relative to $T$.  This allows us to construct a morphism  
\[
\phi_{\mu}:\ce\aro\cm/\ce
\] 
of $\cd(Y/T)$-modules. By assumption, $\phi_{\mu}=0$; we conclude that $\ce$ is invariant under $\cd^1(Y/k)$ (we adopt the notation of \cite[$\mathrm{IV}_4$, 16.8.3, p.40]{EGA}).  

We now assume that $\ce$ is invariant under $\cd^{\ell}(Y/k)$. 
Pick any $\mu=(\mu_1,\ldots,\mu_m)\in\mathbf N^m$ such that $\mu_1+\ldots+\mu_m=\ell+1$. Define, as before, the map of additive groups:
\[
\phi_{\mu}:\ce(U)\aro\cm(U)/\ce(U),\quad e\longmapsto\partial^{[\mu,0]}(e)\mod\ce(U).
\]
If $e$ is a section of $\ce$ over $U$ and 
$\lambda=(\lambda_1,\ldots,\lambda_m)\in\mathbf N^m$ is such that $\lambda_1+\ldots+\lambda_m\le \ell$, then $\partial^{[\lambda,0]}(e)$ is, by hypothesis, also a section of $\ce$. Then, 
\[\begin{split}
\phi_{\mu}(ae)& = \partial^{[\mu,0]}(ae)  \mod\ce(U)\\
&=\sum_{\mu'+\mu''=\mu}\binom{\mu}{\mu'} \partial^{[\mu',0]}(a)\cdot\partial^{[\mu'', 0]}(e)\mod\ce(U)\\
&= a\cdot\partial^{[\mu,0]}(e)\mod\ce(U)\\
&=a\phi_{\mu}(e)\mod\ce(U).
\end{split}
\]
This shows that $\phi_\mu$ is $A$-linear. Note also that \cite[16.11.2, p.54]{EGA} gives
\[
\phi_\mu(\partial^{[0,r]}e)=\partial^{[0,r]}(\phi_\mu(e)).
\]
Finally, $\phi_\mu$ is independent of the choice of the system of coordinates $x_1,\ldots,x_n$ of $U$ relative to $T$. To verify this claim we adopt the same notations as before: for any other system $\overline x_1,\ldots,\overline x_n$ of etale coordinates relative to $T$, we associate the differential operators $\overline \partial^{[q,r]}$, and note that $\overline \partial^{[\mu,0]}-\partial^{[\mu,0]}$ preserves $\ce(U)$, so that $\phi_\mu$ is independent of $x_1,\ldots,x_n$. Hence, we obtain a morphism of $\cd(Y/T)$-modules 
\[
\phi_\mu:\ce\aro\cm/\ce,
\]
which must therefore vanish. Consequently, since $\mu$ is arbitrarily chosen, $\ce$ is invariant under $\cd^{\ell+1}(Y/k)$, and by induction we are done. 
\end{proof}

To deduce from Theorem \ref{21.10.2016--3} the result we need in order to prove Theorem \ref{lem.eh},  Corollary \ref{21.10.2016--4}, we make a few minor observations.

\begin{lem}\label{21.10.2016--2}Let $v:Y'\to Y$ be a schematically dominant \cite[$\mathrm{IV}_3$, 11.10.2,p.171]{EGA} affine morphism of schemes. Let $\mathcal F$ be a locally free $\co_Y$-module. Then $v^*:\Gamma(Y,{\mathcal F})\to\Gamma(Y',v^*{\mathcal F})$ is injective. \qed
\end{lem}

\begin{lem}\label{21.10.2016--1}Assume in addition that $f:X\to S$ is separated. Let $\theta:\ce\to\cm$ be a morphism of coherent $\mathcal O_{X}$-modules and assume that on some finite open covering $\{U_i\}$ of $X$ we can find retractions $r_i:\cm|_{U_i}\to\ce|_{U_i}$. (In particular, $\theta$ is a monomorphism of $\co_X$-modules.) Then, there exists a model $\bm X$ of $X$, a morphism of coherent $\co_{\bm X}$-modules $\bm\theta:\bm\ce\to\bm\cm$, an open covering $\{\bm U_i\}$ of $\bm X$, and retractions $\bm r_i:\bm\cm|_{\bm U_i}\to\bm\ce|_{\bm U_i}$ which, when pulled back by the canonical arrow $X\to\bm X$, give back $\theta$, $\{U_i\}$, and the $\{r_i\}$.
\end{lem}

\begin{proof}All carefully explained in \cite[$\mathrm{IV}_3$]{EGA}, to which we refer in the following lines. From 8.10.5(v) (p.37), we can assume that all models are also separated. Using 8.5.2 (p.20) we obtain our $\bm\theta:\bm\ce\to\bm\cm$. Using 8.2.11 (p.11) we find the covering $\{\bm U_i\}$. Note that, employing 8.2.5 (p.9), the scheme $U_i$ is the limit of the inverse images of $\bm U_i$ in the thinner models. 
To end, another application of 8.5.2 (p.20) gives the retractions. 
\end{proof}

\begin{lem}\label{20.10.2016--1} Suppose  in addition that $f:X\to S$ is separated. Let $\bm\cm\in\dmod{\bm X/\bm S}$ and let $\cm$ stand for the $\cd(X/S)$-module induced by $\boldsymbol{\mathcal M}$. Assume that $\mathcal M$ is a vector bundle and let    
\[
\theta:\mathcal E\aro\mathcal M
\] 
be a special monic of $\dmod{X/S}$. Then there exists, possibly after passing to a thinner model of $X$, a morphism   
\[
\boldsymbol\theta:\bm\ce\aro\bm\cm
\] 
in $\dmod{\bm X/\bm S}$ giving back $\theta:\mathcal E\to\mathcal M$. In addition, the morphism of $\mathcal O_{\bm X}$-modules $\bm\theta$ can be chosen to be injective and to have a locally free cokernel.  
\end{lem}

\begin{proof}Using Lemma \ref{21.10.2016--1}, we can find a morphism of $\mathcal O_{\bm X}$-coherent modules $\bm\theta:\bm\ce\to\bm\cm$ giving back $\theta$ and possessing, locally, a retraction. Moreover, employing \cite[$\mathrm{IV}_3$, 8.5.5,p.23]{EGA} it is licit to assume that $\bm \ce$ and $\bm \cm$ are locally free.

Passing to a thinner model, we can assume that $\theta$ comes from an arrow $\boldsymbol\theta:\boldsymbol{\mathcal E}\to\boldsymbol{\mathcal M}$ of locally free $\co_{\bm X}$-modules, see \cite[$\mathrm{IV}_3$]{EGA},  8.5.2(i), p.20 and 8.5.5, p.23. Moreover, Lemma \ref{21.10.2016--1} allows us to assume that locally $\boldsymbol\theta$ possesses a retraction. We now need to show that $\bm\ce$ can be given a structure of a $\cd(\bm X/\bm S)$-module, and for that we only need to prove that $\bm\ce$ is locally invariant under the action of $\cd(\bm X/\bm S)$.

So let
$\bm U\subset \bm X$ be an affine open of $\bm X$ having the following properties: (a) the arrow $\bm\theta(\bm U):\bm\ce(\bm U)\to\bm\cm(\bm U)$ possesses a retraction $\bm r:\bm\cm(\bm U)\to\bm\ce(\bm U)$, (b) the $\co(\bm U)$-module $\bm\cm(\bm U)$ is free, and (c) on $\bm U$ we have etale coordinates $\bm x_1,\ldots,\bm x_n$ with respect to $\bm S$. Moreover, we write  $U=\bm U\ti_{\bm S}S$. 

In this setting, we conclude that, if $1\ot\bm m\in\co(U)\ot_{\co(\bm U)}\bm\cm(\bm U)$ actually lies in $\co(U)\ot_{\co(\bm U)}\bm\ce(\bm U)$, then $1\ot\bm\theta\bm r(\bm m)=1\ot\bm m$, which shows that  $\bm m\in\bm\ce(\bm U)$ since $\bm\cm(\bm U)\to\co(U)\ot_{\co(\bm U)}\bm\cm(\bm U)$ is injective. Consequently, if $\bm\partial^{[q]}$ is the differential operator on $\bm U$ constructed via the above  system of coordinates \cite[$\mathrm{IV}_4$, 16.11.2, p.54]{EGA} we see that $\bm\partial^{[q]}(\bm e)\in\bm\ce(\bm U)$ for each $\bm e\in\bm\ce(\bm U)$ as $1\ot\bm\partial^{[q]}(\bm e)$ does belong to $\co(U)\ot_{\co(\bm U)}\bm\ce(\bm U)$.
\end{proof}

\begin{cor}\label{21.10.2016--4}We suppose that $f$ is separated. Let $\bm\cm\in\dmod{\bm X/k}$  and consider a special monic $\theta:\ce\to\mathrm{Inf}_X(\bm\cm)$ such that 
\[
\mathrm{Hom}_{\cd(X/S)}(\ce,\mathrm{Inf}_X(\bm\cm)/\ce)=0.
\]
Then, replacing eventually $\bm X$ by some thiner model of $X$, there exist $\bm\ce\in\dmod{\bm X/k}$ and an arrow $\bm\theta:\bm\ce\to\bm\cm$ which is taken to $\theta$ by the functor $\mathrm{Inf}_X(-)$.
\end{cor}

\begin{proof}Lemma \ref{20.10.2016--1} shows that $\theta:\ce
\to\mathrm{Inf}_X(\bm\cm)$ is the image of an arrow $\bm\ce\to\mathrm{Inf}(\bm\cm)$ between $\cd(\bm X/\bm S)$-modules.  Moreover,  
$\bm\theta$ is injective and its cokernel, call it $\bm{\mathcal C}$, is locally free. 
We now show that   
\[
\mathrm{Hom}_{\cd(\bm X/\bm S)}(\bm\ce\,,\,\bm{\mathcal C})=0.
\]
Let 
$\bm\varphi:\bm\ce\to\bm{\mathcal C}$ be a morphism of $\cd(\bm X/\bm S)$-modules. Then, if $u:X\to \bm X$ stands for the canonical morphism, we see that  $u^*(\bm\varphi)=:\varphi$ is a morphism of $\cd(X/S)$-modules from $\ce$ to the cokernel of $\theta$.
By assumption, $\varphi=0$ and because of Lemma \ref{21.10.2016--2},  $\bm\varphi=0$. 

We can now apply Theorem \ref{21.10.2016--3} to show that $\bm\theta$ is an arrow of   $\dmod{\bm X/k}$.
\end{proof}

We can now present the

\begin{proof}[Proof of Theorem \ref{lem.eh}]The idea is to apply Corollary \ref{21.10.2016--4} to the socle series obtained from Lemma \ref{24.02.2015--3}. We begin by assuming that $\mathcal V$ is of rank one as an $\mathcal O_X$-module. 

Employing the terminology of  Lemma \ref{24.02.2015--3},  define $\mathrm{Soc}_1(\mathcal V)$ as the $\mathcal V$-socle of $\mathcal M$. If $\mathrm{Soc}_i$ is defined, put 
\[
\mathrm{Soc}_{i+1}(\mathcal V)=\begin{array}{cc}\text{inverse image in $\mathcal M$ of} \\ \text{the $\mathcal V$-socle of $\mathcal M/\mathrm{Soc}_i(\mathcal V)$,}\end{array}
\]
so that \[\mathrm{Soc}_{i+1}(\mathcal V)/\mathrm{Soc}_i(\mathcal V)\simeq \text{$\mathcal V$-socle of $\cm/\mathrm{Soc}_i(\cv)$}.\]
Using that the socle is a special subobject (Lemma \ref{24.02.2015--3}) we see that $\mathrm{Soc}_i(\mathcal V)\subset\mathcal M$ is special too, so, for some $r\in\mathbf N$, we have 
\[  \mathrm{Soc}_{r}(\mathcal V)= \mathrm{Soc}_{r+1}(\mathcal V)=\cdots\]  Due to the definition of the socle, we conclude that  there are no submodules of 
\[
\mathcal M/\mathrm{Soc}_r(\mathcal V)
\]
isomorphic to $\mathcal V$. The assumption on the rank of $\mathcal V$ and Proposition \ref{05.03.2014--1} then force  all arrows  
\[
\mathcal V\aro  \mathcal M/\mathrm{Soc}_r(\mathcal V)
\]
in $\dmod{X/S}$ to be null. 
Using that 
\[
\mathrm{Soc}_1(\mathcal V),\quad  \frac{\mathrm{Soc}_{2}(\mathcal V)}{\mathrm{Soc}_{1}(\mathcal V)},\quad \ldots
\]
are all isomorphic to direct sums of $\mathcal V$, we see that any arrow 
\[
\mathrm{Soc}_r(\mathcal V)\aro \mathcal M/\mathrm{Soc}_r(\mathcal V)
\]
is null. By Corollary \ref{21.10.2016--4}, the $\cd(X/S)$-module $\mathrm{Soc}_r(\mathcal V)$ is an inflation, i.e. after eventually passing to a thiner model,  
there exists  $\bm \cs\in\dmod{\bm X/k}$ together with a monomorphism $\bm\cs\to\bm\cm$ of $\cd(\bm X/k)$-modules such that 
\[\mathrm{Inf}_X(\bm\cs)\aro \mathrm{Inf}_X(\boldsymbol{\mathcal M})=\mathcal M\]
is our special monic $\mathrm{Soc}_r(\mathcal V)\to\mathcal M$. Since  $\mathcal V$ is a quotient of  $\mathrm{Soc}_r(\mathcal V)=\mathrm{Inf}_X(\boldsymbol{\mathcal S})$ and since $\boldsymbol{\mathcal S}$ is a subobject of $\bm\cm$, we are done. 

In general, write  $m=\mathrm{rank}(\mathcal V)$ and let  $\mathrm{Inf}_X(\bm\cs)\to\det(\cv)$
be an epimorphism. Then
\[
\left\{\wedge^{m-1}\mathcal V\right\}^\vee\ot \det(\mathcal V)\simeq \mathcal V,
\]
which shows that $\mathcal V$ is a quotient of $\mathcal M^\vee\ot  \mathrm{Inf}_X(\bm\cs)$.
\end{proof}

\end{document}